\documentclass[reqno]{brandiss}

\usepackage{graphicx}

\usepackage{geometry} 
\usepackage{amssymb, latexsym, amsmath, graphicx, amsthm, enumitem}
\usepackage{amscd}
\usepackage{url}
\usepackage[small,nohug,heads=littlevee]{diagrams}
\diagramstyle[labelstyle=\scriptstyle]

\numberwithin{section}{chapter}
\numberwithin{figure}{chapter}

\disstitle{Generalized Chung-Feller Theorems for Lattice Paths}
\dissauthor{Aminul Huq}
\dissadvisor{Ira M. Gessel}
\dissdepartment{Mathematics}
\dissmonth{August}
\dissyear{2009}
\dissdean{Adam Jaffe}

\newtheorem{theorem}{Theorem}[section]
\newtheorem{cor}[theorem]{Corollary}
\newtheorem*{lemma}{Lemma}

\theoremstyle{remark}

\theoremstyle{definition}
\newtheorem{defn}[theorem]{Definition}
\numberwithin{equation}{section}

\DeclareMathOperator{\dr}{dr}
\DeclareMathOperator{\pk}{pk}

\DeclareMathOperator{\df}{df}

\DeclareMathOperator{\pr}{MP}
\allowdisplaybreaks

\begin{document}


\frontmatter

\makedisstitle



\disscopyright 

\begin{dissdedication}
\vspace{2cm}
\begin{center}
To My Parents
\end{center}
\end{dissdedication}

\begin{dissacknowledgments} 
I wish to express my heartful gratitude to my advisor, Professor Ira M. Gessel, for his teaching, help, guidance, patience, and support.

I am grateful to the members of my dissertation defense committee Professor Richard P. Stanley and Professor Susan F. Parker. Specially I'm greatly indebted to Professor Parker for her continual encouragement and mental support. I learned a great deal from her about teaching and mentoring. 
  
I owe thanks to the faculty, specially Professor Mark Adler and Professor Daniel Ruberman,  my dear friend Apratim Roy for his helpful suggestions, to my fellow students, and to the kind and supportive staff of the Brandeis Mathematics Department.

I would like to thank all my family and friends for their love and encouragement with patience and I wish to express my boundless love to my wife, Arifun Chowdhury.

This thesis is dedicated to my parents, Md. Enamul Huq and Mahbub Ara Ummeh Sultana, with my deep gratitude. 
\end{dissacknowledgments}

\begin{dissabstract}
In this thesis we develop generalized versions of the Chung-Feller
theorem for lattice paths constrained in the half plane. The beautiful cycle method which was developed by Devoretzky and Motzkin as a means to prove the ballot problem is modified and applied to generalize the classical Chung-Feller theorem. We use Lagrange inversion to derive the generalized formulas. For the generating function proof we study various ways of decomposing lattice paths. We also show some results related to equidistribution properties in terms of Narayana and Catalan generating functions. 
We then develop generalized Chung-Feller theorems for Motzkin and Schr\"oder paths. Finally we study generalized paths and the analogue of the Chung-Feller theorem for them.
\end{dissabstract}


\tableofcontents 

\listoffigures 


\mainmatter

\chapter{Introduction}
In discrete mathematics, all sorts of constrained lattice paths serve to describe apparently complex objects. The simplest lattice path problem is the problem of counting paths in the plane, with unit east and north steps, from the origin to the point $(m, n)$. The number of such paths is the binomial coefficient $\binom{m+n}{n}$. We can find more interesting problems by counting these paths according to certain parameters like the number of left turns (an east step followed by a north step), the area between the path and the $x$-axis, etc. If $m = n$ then the classical Chung-Feller theorem \cite{hh} tells us that the number of such paths with $2k$ steps above the line $x = y$ is independent of $k$, for $k = 0, \dots, n$ and is therefore equal to the Catalan number $C_n =\tfrac{1}{n+1} \binom{2n}{n}$. The simplest, and most fundamental, result of lattice paths constrained in a subregion of the plane is the solution of the ballot problem: the number of paths from $(1, 0)$ to $(m, n)$, where $m > n$, that never touch the line $x = y$, is the ballot number $\tfrac{m-n}{m+n} \binom{m+n}{n}$. In the special case $m = n + 1$, this ballot number is the Catalan number $C_n$. The corresponding paths are often redrawn as paths with northeast and southeast steps that never go below the $x$-axis; these are called Dyck paths:

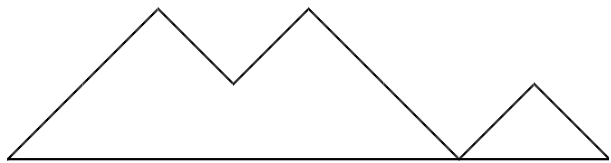
\begin{figure}[ht]
\begin{center}
\ifx\JPicScale\undefined\def\JPicScale{1}\fi
\unitlength \JPicScale mm
\begin{picture}(90,30)(0,0)
\linethickness{0.4mm}
\multiput(10,10)(0.12,0.12){83}{\line(1,0){0.12}}
\linethickness{0.4mm}
\multiput(20,20)(0.12,0.12){83}{\line(1,0){0.12}}
\linethickness{0.4mm}
\multiput(30,30)(0.12,-0.12){83}{\line(1,0){0.12}}
\linethickness{0.4mm}
\multiput(40,20)(0.12,0.12){83}{\line(1,0){0.12}}
\linethickness{0.4mm}
\multiput(50,30)(0.12,-0.12){83}{\line(1,0){0.12}}
\linethickness{0.4mm}
\multiput(60,20)(0.12,-0.12){83}{\line(1,0){0.12}}
\linethickness{0.4mm}
\multiput(70,10)(0.12,0.12){83}{\line(1,0){0.12}}
\linethickness{0.4mm}
\multiput(80,20)(0.12,-0.12){83}{\line(1,0){0.12}}
\linethickness{0.3mm}
\put(10,10){\line(1,0){80}}
\end{picture}
\vspace{-.8cm}
	\caption{A Dyck path}
\end{center}
\end{figure}

Dyck paths are closely related to traversal sequences of general and binary trees; they belong to what Riordan has named the ``Catalan domain", that is, the orbit of structures counted by the Catalan numbers. The wealth of properties surrounding Dyck paths can be perceived when examining either Gould's monograph \cite{gould} that lists $243$ references or from Exercise 6.19 in Stanley's book \cite{dd} whose statement alone spans more than $10$ full pages.

The classical Chung-Feller theorem was proved by Major Percy A. MacMahon in 1909 \cite{mc}. 
Chung and Feller reproved this theorem by using the generating function method in \cite{hh} in 1949. T. V. Narayana \cite{kk} showed the Chung-Feller theorem by combinatorial methods. Mohanty's book \cite{moh} devotes an entire section to exploring the Chung-Feller theorem. S. P. Eu et al. \cite{ff} proved the Chung-Feller Theorem by using Taylor expansions of generating functions and gave a refinement of this theorem. In \cite{ii}, they gave a strengthening of the Chung-Feller theorem and a weighted version for Schr\"oder paths. Both results were proved by refined bijections which are developed from the study of Taylor expansions of generating functions. Y. M. Chen \cite{gg} revisited the Chung-Feller theorem by establishing a bijection. David Callan in \cite{dc} and R. I. Jewett and K. A. Ross in \cite{jr} also gave bijective proofs of the Chung-Feller theorem. J. Maa and Y.-N. Yeh studied Chung-Feller Theorem for the non-positive length and the rightmost minimum length in \cite{ma-2009}.

Therefore generalizations of the Chung-Feller theorem have been visited by several authors as described above. But the most interesting aspect of the Chung-Feller theorem was the interpretation of the Catalan number formula $\tfrac{1}{n+1} \binom{2n}{n}$ that explained the appearence of the fraction $\tfrac{1}{n+1}$. However there are two other equivalent forms of the Catalan number formula which do not fit into the classical version of the Chung-Feller theorem. Moreover there are several other kinds of lattice paths like Motzkin paths, Schr\"oder paths, Riordan paths, etc. and associated number formulas and equivalent forms that have not been studied using generalized versions of the Chung-Feller theorem. 

The same can be said about their higher-dimensional versions \cite{higherdim} and $q$-analogues. 
For that reason the main purpose of this thesis is to find more systematic generalizations of the Chung-Feller theorem. We apply the cycle method to this problem. 

In the next section we present the classical Chung-Feller theorem along with the definitions and notations that we'll use. In chapter two we give the modified cycle method and the notion of special vertices and use that  to derive the generalized Chung-Feller theorems for Catalan and Narayana number formulas. Chapter three deals with generalized Chung-Feller theorems for Motzkin, Schr\"oder, and Riordan number formulas. In chapter four we use generating functions to prove generalized Chung-Feller theorems for Catalan and Narayana numbers and also describe the equidistribution property of left-most highest points and up steps in even positions for paths that end at height one and height two respectively. In chaper five we develop generalized Chung-Feller theorems for generalized Catalan and Narayana number formulas. 

\clearpage

\section{Lattice paths and the Chung-Feller theorem}\label{sec:first}

	In this section we present the varieties of lattice paths to be studied and restate the Chung-Feller theorem with proofs. We begin with the formal definition of the paths that we will be dealing with. 

\begin{defn} Fix a finite set of vectors in $\mathbb{Z} \times \mathbb{Z}$, $\mathcal{V}= \{(a_1,b_1), \dots, 	
	(a_m,b_m)  \}$. A {\it lattice path} with steps in $\mathcal{V}$ is a sequence $v = (v_1, \dots, v_n )$ such that 
	each $v_j$ is in $\mathcal{V}$. The geometric realization of a lattice path $v = (v_1, \dots, v_n )$ is the 
	sequence of points $(P_0, P_1, \dots, P_n )$ such that $P_0=(0,0)$ and 
	$P_i - P_{i-1} = v_i$. The quantity $n$ is referred to as the length
of the path.
\end{defn}
 
	In the sequel, we shall identify a lattice path with the polygonal line admitting $P_0, P_1, \dots, P_n $  as
	vertices. The elements of $\mathcal{V}$ are called steps, and we also refer to the vectors 
	$P_{i} - P_{i-1} = v_i$ as the steps of a particular path. Various constraints will be imposed on paths. We 
	consider the following condition on the paths we'll concern
ourselves with. 

\begin{defn}
	Let $\mathcal{P}(n, r, h)$ be the set of paths (referred to simply as {\it paths}) having the step set 
	$S = \{(1,1), (1,-r)\}$ that lie in the half plane $\mathbb{Z}_{\ge 0}\times \mathbb{Z}$ ending at
	$((r+1)n+h,h)$, where we call $n$ the semi-length. We denote by
$\mathcal{P}(n, 1, 0, +)$ the paths in 
	$\mathcal{P}(n, 1, 0)$ that lie in the quarter plane 
	$\mathbb{Z}_{\ge 0}\times \mathbb{Z}_{\ge 0}$. They are known as {\it Dyck paths} (we'll also refer to them
	as {\it positive paths}). We also denote by 
	$\mathcal{P}(n, 1, 0, -)$ the set of {\it negative paths} which 
	are just the reflections of $\mathcal{P}(n,1,0,+)$ about the $x$-axis. 
\end{defn}

	A lot of effort has been given to enumerating the above mentioned paths according to different parameters 
	and with restrictions. We know that the total number of Dyck paths of length $2n$ is given by the Catalan 
	number $C_n$ and the well known Chung-Feller theorem \cite{hh}, stated below, gives a nice combinatorial 
	interpretation for the Catalan number formula which generalizes the enumeration of Dyck paths.  

\begin{theorem}{\bf (Chung-Feller)} Among the $\binom{2n}{n}$ paths from $(0,0)$ to $(2n,0)$, the number of 
	paths with $2k$ steps lying above the $x$-axis is independent of $k$ for $0 \le k \le n$, and is equal to $\tfrac{1}{n+1} \binom{2n}{n}$.
\end{theorem}
The Chung-Feller theorem only deals with paths having steps of the form 
	$(1,1)$ and $(1,-1)$ whereas the cycle lemma, first introduced by Dvoretzky and Motzkin \cite{ll}, gives us an indication that a generalized Chung-Feller theorem might exist that can take into account more general paths. 

	If we let $k=n$ so that all the steps lie above the $x$-axis then we just get the Dyck paths. There are two 
	other equivalent expressions for the Catalan number 
	$C_{n}\!: \tfrac{1}{2n+1} \binom{2n+1}{n} $ and $\tfrac{1}{n} \binom{2n}{n-1}$, which await similar combinatorial interpretations. David Callan \cite{dcal} gave an interpretation of these forms using paths that end at different heights. In the next section we give a general method for explaining formulas like this. In all cases we count paths that end at $(2n+1, 1)$. Our interpretation shows that the formula $ \tfrac{1}{2n+1} \binom{2n+1}{n} $ corresponds to counting all such paths according to the number of points on or below the $x$-axis, $\tfrac{1}{n+1} \binom{2n}{n}$ corresponds to counting such paths starting with an up step according to the number of up steps starting on or below the $x$-axis and $\tfrac{1}{n} \binom{2n}{n-1}$ corresponds to counting such paths starting with a down step  according to the number of down steps starting on or below the $x$-axis.


\clearpage

\chapter{A generalized Chung-Feller theorem}
\section{The cycle method}
An important method of counting lattice paths is the ``cycle lemma" of Dvoretzky and Motzkin \cite{ll}. It may be stated in the following way: For any $n$-tuple $(a_1 , a_2 , \dots, a_n )$ of integers from the set $\{1, 0, -1, -2, \dots \}$ with sum $k > 0$, there are exactly $k$ values of $i$ for which the cyclic permutation $(a_i , \dots, a_n , a_1 , \dots, a_{i-1} )$ has every partial sum positive. The special case in which each $a_i$ is either $1$ or $-1$ gives the solution to the ballot problem. The Chung-Feller theorem, and some of its generalizations, can be proved by a variation of the cycle lemma. It is worth noting here that Dvoretzky and Motzkin \cite{ll} stated and proved the cycle lemma as a means of solving the ballot problem. Dershowitz and Zaks \cite{zak} pointed out that this is a ``frequently rediscovered combinatorial lemma" and they provide two other applications of the lemma. They stated that the cycle lemma is the combinatorial analogue of the Lagrange inversion formula.

We are going to apply the ``cycle method"  to develop generalized Chung-Feller theorems. This approach was first used by Narayana \cite{kk} in a less transparent way to prove the original Chung-Feller theorem. We'll use sequences instead of paths to prove the theorem to make things easier. We define the cyclic shift $\sigma$ on sequences ${\bf a} = (a_1,a_2,\dots,a_n)$ by
	\[\sigma(a_1,a_2,\dots,a_n)=(a_2,a_3,\dots,a_n,a_1). \]
A conjugate of $(a_1, a_2,\dots, a_n)$ is a sequence of the form 
	\[\sigma^i (a_1, a_2, \dots, a_n) = (a_{i+1}, a_{i+2}, \dots, a_n, a_1, \dots, a_i)\] for some $i$.
With these definition we state a variation of the cycle lemma. 
\begin{theorem} \label{t2}
    	Suppose that $a_1 + a_2 + \cdots +  a_n = 1$ where each $a_i \in \mathbb{Z}, i=1,\dots,n$. Then for each 
	$k$, $1\le k \le n$, there is exactly one conjugate of the sequence ${\bf a}=(a_1,\dots,a_n)$ with exactly $k$ 
	nonpositive partial sums.
\end{theorem}

\begin{proof}
	We define 
	$S_i ({\bf a})$ to be $a_1 + \cdots + a_i -\tfrac{i}{n}$ for $0\le i \le n$. Note that 
	$S_0 ({\bf a})=S_n({\bf a}) = 0$ 
	and it is clear that for $0\le i \le n-1$, $S_i({\bf a}) \le 0$ if and
only if $a_1 + \cdots + a_i \le 0$. Let
	us also define $a_j$ for $j>n$ or $j \le 0$ by setting $a_j=a_i$ whenever 
	$j \equiv i \pmod{n}$. So, $S_i ({\bf a})$ is defined for all $i \in {\mathbb Z}$; i.e., if 
	$j \equiv i \pmod{n} $ then $S_j = S_i ({\bf a})$.

	We observe that since the fractional parts of $S_0({\bf a}), \dots , S_{n-1} ({\bf a})$ are all different, all 
	$S_i ({\bf a}), 0 \le i \le n-1$, are distinct.

	To prove the theorem it is enough to show that if $S_i ({\bf a}) < S_j ({\bf a})$ then $\sigma^j ({\bf a})$ has 
	more nonpositive partial sums than  $\sigma^i ({\bf a})$, since the number of nonpositive partial sums is in 
	$\{1, 2, \dots, n\}$.
Suppose that $S_i ({\bf a}) < S_j ({\bf a})$. Then we have
\begin{equation}
	\label{a} 
	S_k (\sigma^j ({\bf a})) = S_k ((a_{j+1}, \dots, a_n, a_1, \dots, a_j)) =
	a_{j+1} + \cdots+a_{j+k}- \tfrac{k}{n} 
\end{equation}
	and
\begin{equation}
	S_{k+j-i} (\sigma^i ({\bf a})) = S_{k+j-i} (a_{i+1},  \dots, a_n, a_1, \dots,  a_i) = 
	a_{i+1} + \cdots+a_{j+k}- \tfrac{k+j-i}{n}.
\end{equation}
This is true even if $j+k>n$. So,
\begin{align}
    	S_{k+j-i} (\sigma^i ({\bf a})) - S_k (\sigma^j ({\bf a}))  \notag
    	&= ( a_1 + \cdots + a_j- \tfrac{j}{n}) - (a_1 + \cdots + a_i - \tfrac{i}{n})\\ \notag
    	&= S_j ({\bf a}) - S_i ({\bf a})\\
    	&>0.
\end{align}
So if $S_{k+j-i} (\sigma^i ({\bf a})) \le 0$ then $S_{k} (\sigma^j ({\bf a})) < S_{k+j-i} (\sigma^i ({\bf a})) \le 0$. 
Moreover for $k=0$, we have  
\[ S_{j-i} (\sigma^i ({\bf a})) - S_{0} (\sigma^j ({\bf a})) > 0. \]

Since $S_{0} (\sigma^j ({\bf a}))=0$, this shows that $\sigma^j ({\bf a})$ has at least one more nonpositive partial sum than $\sigma^j ({\bf a})$.
\end{proof}

	We can give a geometric interpretation of this result in terms of lattice paths that will make it easier to understand.
	\begin{figure}[h!]
	\centering
 \includegraphics[width=.8\textwidth]{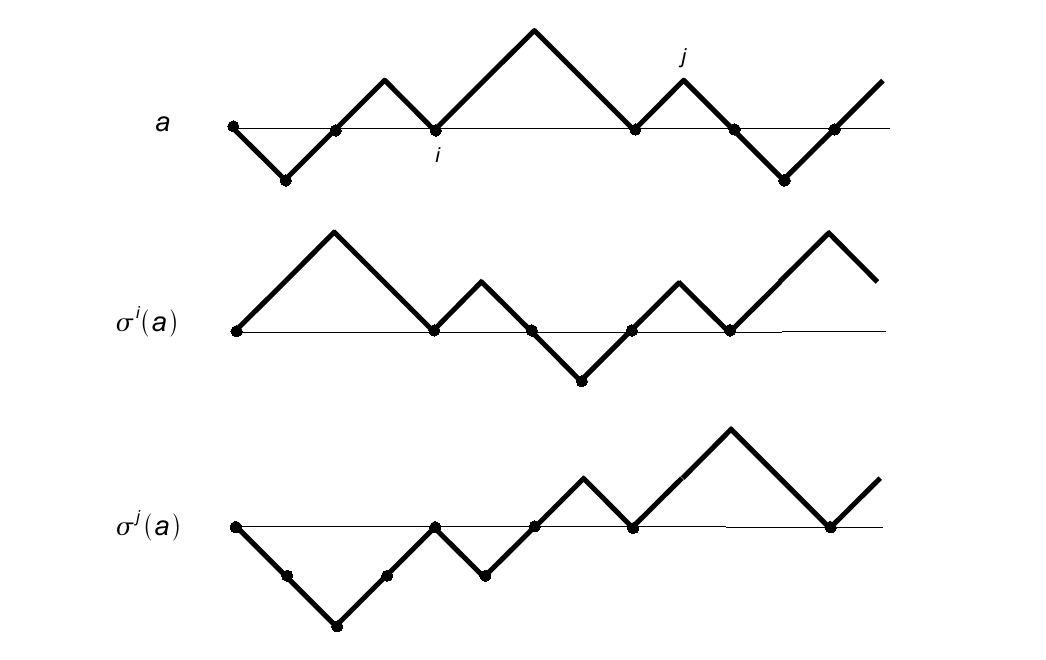}
	\caption{Two cyclic shifts of a sequence {\bf a} represented by a path}
	\label{fig: cyclic shift}
\end{figure}

	We can associate to a sequence $(a_1,a_2,\dots,a_n)$ a path
	$p=(p_1,p_2,\dots,p_n)$ in which $p_i$ is $(1,a_i)$ which is 
	either an up step that goes up by $a_i$,
	a flat step, or a down step that goes down by $-a_i$, whenever
   	$a_i$ is positive, zero, or negative respectively. Since
   	$a_1 + a_2 + \cdots +  a_n = 1$, the path ends at height $1$
   	and the nonpositive partial sums correspond to vertices of the
   	path on or below the $x$-axis. We define a conjugate of a path
   	$p=(p_1,p_2,\dots,p_n)$ to be a path of the form
   	$\sigma^i(p)=(p_{i+1}, \dots, p_{n}, p_1, \dots, p_i)$. 
   	
   	With these definitions a special case of  Theorem \ref{t2} can be stated as follows:
\begin{theorem} \label{t3}	For a path in $\mathcal{P}(n, 1, 1)$ that starts at $(0,0)$ and ends at
	height $1$ there is exactly one conjugate of $p$ with
	exactly $k$ vertices on or below the $x$-axis for each $k$, $1\le
	k \le n$.
\end{theorem}

Figure \ref{fig: cyclic shift} illustrates the nonpositive sums given in the proof as the vertices of the path on or below the $x$-axis.

\clearpage
\section{Special vertices}\label{sec:second}
		We can extend $\sigma$ in a natural 
	way to the vertices of paths, so that a vertex $v$ of a path $p$ 
	corresponds to the vertex $\sigma^j (v)$ of the path  $\sigma^j (p)$.
	For each path $p$ we take a subset of the vertex set of $p$ 
	which we call the set of {\it special vertices} of $p$. We require 
	that special vertices are preserved by cyclic permutation, so that $v$ 
	is a special vertex of $p$ if and only if $\sigma^j (v)$ is a special 
	vertex of $\sigma^j (p)$. Unless otherwise stated we will not include 
	the last vertex as a special vertex.

\begin{theorem} \label{spver} 
	Suppose $p$ has $k$ special vertices.
	Let $\sigma^{t_1}(p),\dots, \sigma^{t_{k}}(p)$ be the $k$ conjugates 
	of $p$ that start with a special vertex. For each such path let $X(\sigma^i (p))$
	be the number of special vertices on or below the $x$-axis. Then
\begin{equation}
	\{ X(\sigma^{t_1} (p)), X(\sigma^{t_2} (p)), \dots, 
	 X(\sigma^{t_{k}} (p)) \} = \{1, 2, \dots, k\}.
\end{equation}
\end{theorem}

\begin{proof}
	Given a sequence ${\bf a}$ as in Theorem 2, let the sequence
	${\bf b} = (b_1, b_2, \dots, b_k)$ be defined by
\begin{align}
	b_1 &= a_1 +\dots + a_{t_1} \notag \\ \notag
         b_2 &= a_{t_1+1} +\dots + a_{t_2}\\ \notag
                 & \vdots\\ 
         b_m &= a_{t_{m-1}+1}+\dots+a_{t_m}
\end{align}
	where $t_1<t_2<\dots<t_m=n.$ Since $b_i \in {\mathbb Z}$ and
	$\sum_{i=1}^{m} b_i = 1$ by Theorem \ref{t2} we have that for each
	$k$, $1 \le k \le m$, there is exactly one conjugate of ${\bf
	b}$ with exactly $k$ nonpositive partial sums.
\end{proof}

\clearpage
\section{The three versions of the Catalan number formula}\label{sec:third}
		We can use the notion of special vertices and Theorem \ref{spver} to
give a nice combinatorial interpretation to the three versions of the Catalan
number formula as follows:

\begin{theorem} \label{t5} \  
\begin{enumerate}
\item The number of paths in $\mathcal{P}(n, 1, 1)$ that start with an up step with 
	exactly $k$ up steps starting on or below the $x$-axis for $k = 1, 2, \dots, n+1$ is $\tfrac{1}{n+1}\binom{2n}{n}$.

\item The number of paths in $\mathcal{P}(n, 1, 1)$ that start with a down step with 
	exactly $k$ down steps that start on or below the $x$-axis for $k=1, 2 , \dots, n$ 
	is $\tfrac{1}{n} \binom{2n}{n-1}$.

\item The number of paths in $\mathcal{P}(n, 1, 1)$ with exactly $k$ vertices on or 
	below the $x$-axis for $k = 1, 2, \dots, 2n+1$ is $\tfrac{1}{2n+1} \binom{2n+1}{n}$.
\end{enumerate}
\end{theorem}

\begin{proof}
	This is just a straightforward application of Theorem \ref{spver}. 
	First we'll prove the first part. Let $p$ be any path in $\mathcal{P}(n, 1, 1)$. 
	So $p$ starts from $(0, 0)$ and ends at $(2n+1,1)$ with $n+1$ up steps and $n$ down 
	steps. We take the initial 
	vertices of the up steps of $p$ as our special vertices. Since there 
	are $n+1$ up steps, $p$ has $n$ conjugates that start with an up step. 
	By Theorem \ref{spver} there is exactly one conjugate of $p$ with exactly 
	$k$ up steps starting on or below the $x$-axis and we know that the
number of paths 
	in $\mathcal{P}(n, 1, 1)$ that start with an up step is given by the binomial 
	coefficient $\binom{2n}{n}$. Therefore the number of paths starting with 
	an up step and having $k$ up steps on or below the $x$-axis is given by 
	$\tfrac{1}{n+1}\binom{2n}{n}$ as stated in part one. The proof of part two is similar, where 
	we consider the initial vertices of the down steps as special vertices. 
	For part three we consider the initial vertices of all the steps as special 
	vertices and use the same argument. 
\end{proof}

Note that part one of the theorem is basically the classical Chung-Feller
theorem. To make the connection we just need to remove the first up step and
lower the path one level down. Then we get a path in $\mathcal{P}(n, 1, 0)$ that
starts and ends on the $x$-axis with $k$ up steps starting below the $x$-axis.
Since the number of up and down steps below the $x$-axis are the same, having $k$
up steps below the $x$-axis is the same as having $2k$ steps below the $x$-axis.
 
 \clearpage
\section{Words}\label{sec:fourth}
	We can encode each up step by the letter $U$ (for up) and each down 
	step by the letter $D$ (for	down), obtaining the encoding of paths
in 
	$\mathcal{P}(n, 1, 1)$ as {\it words}. For example,
	the path in Fig. \ref{fig:lattice} is encoded by the word \[ UDUDDDUDUUDUUUDUD.\] 
	In a path a {\it peak} is an occurrence of $UD$, a {\it valley} is an occurrence of $DU$,
	a {\it double rise} is an occurrence of $UU$, and a {\it double fall} is an occurrence of $DD$.
\begin{figure}[ht]
	\centering
		\includegraphics[width=0.80\textwidth]{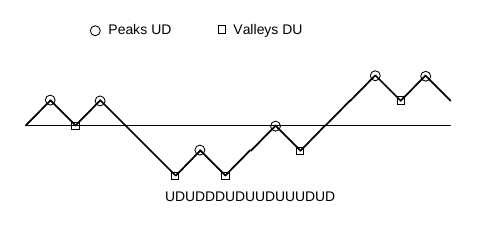}
	\caption{Peaks and valleys}
	\label{fig:lattice}
\end{figure}
	
	By a peak lying on or below the $x$-axis we mean the vertex between the
up step and the down step lying on or below the
	 $x$-axis and for a double rise we consider the vertex between the two
consecutive up steps lying on or below the $x$-axis. Similarly for valleys and double falls.
In the next section we'll count paths according to the number of these special vertices lying on or below the $x$-axis.
\clearpage
\section{Versions of the Narayana number formula}\label{sec:fifth}
\begin{defn}
  The Narayana number $N(n,k)$ \cite{narayana} counts Dyck paths from
    	$(0,0)$ to $(2n,0)$ with $k$ peaks and is given by
\begin{equation*}
		N(n,k) = \dfrac{1}{n} \binom{n}{k} \binom{n}{k-1}
\end{equation*} for $n \ge 1.$
    	$N(n,k) $ can also be expressed in five other forms as
\begin{eqnarray*}
N(n,k) &=&  \dfrac{1}{k} \binom{n}{k-1} \binom{n-1}{k-1}
		  = \dfrac{1}{n-k+1} \binom{n}{k} \binom{n-1}{k-1}\\
		   &=& \dfrac{1}{n+1} \binom{n+1}{k} \binom{n-1}{k-1}
   	       = \dfrac{1}{k-1} \binom{n}{k} \binom{n-1}{k-2}\\
		&=& \dfrac{1}{n-k} \binom{n}{k-1} \binom{n-1}{k}.
\end{eqnarray*}
\end{defn}

	These numbers are well known in the literature since they have many 
	combinatorial interpretations (see for example Sulanke \cite{cc}, which 
	describes many properties of Dyck paths having the Narayana distribution). Deutsch \cite{aa} studied the enumeration of Dyck paths according to various parameters, several of which involved Narayana numbers.
	
	The generalized Chung-Feller theorem can also be used to give combinatorial 
	interpretation of the different versions of the Narayana number formula 
	taking the special vertices as peaks, valleys, double rises, and double falls.
	
\begin{theorem}\label{t6} \ 
\begin{enumerate}
\item  The number of paths in  $\mathcal{P}(n, 1, 1)$ with $k-1$ peaks that start with a down step and end with an up step with exactly $j$ peaks on or 
	below the $x$-axis for $j = 0, 1, 2, \dots, k-1$ is given by $\tfrac{1}{k} \binom{n}{k-1} \binom{n-1}{k-1}$.

\item The number of paths in  $\mathcal{P}(n, 1, 1)$ with $k-1$ valleys
	that start with an up step and end with a down step with exactly 
	$j$ valleys on or below the $x$-axis for $j=0, 1, 2, \dots, k-1$  is given by 
    	$\tfrac{1}{k} \binom{n}{k-1} \binom{n-1}{k-1} $.

\item The number of paths in  $\mathcal{P}(n, 1, 1)$ with 
	$n-k$ double rises that start with an up step and end with an up step
with exactly $j$ double 
	rises on or below the $x$-axis for $j=0, 1, 2, \dots, n-k$ is given by 
	$\tfrac{1}{n-k+1} \binom{n}{k}  \binom{n-1}{k-1}$.

\item The number of paths in  $\mathcal{P}(n, 1, 1)$ with 
 	$n-k-1$ double falls that start with a down step and end with a down
step with exactly 
 	$j$ double falls on or below the $x$-axis for $j=0, 1, 2, \dots, n-k-1$ is given by 
 	$\tfrac{1}{n-k} \binom{n}{k-1}  \binom{n-1}{k}$.

\item The number of paths in  $\mathcal{P}(n, 1, 1)$ with $k$ peaks that start with 
	an up step with exactly $j$ up steps starting on or below the $x$-axis for $j= 1, 2, \dots, n+1$ 
	is given by $\tfrac{1}{n+1} \binom{n+1}{k}  \binom{n-1}{k-1}$.

\item The number of paths in  $\mathcal{P}(n, 1, 1)$ with $k$ valleys that start 
	with a down step with exactly $j$ down steps starting on or below the $x$-axis for 
	$j= 1, 2, \dots, n$ is given by $\tfrac{1}{n} \binom{n}{k}  \binom{n}{k-1}$.
\end{enumerate}
\end{theorem}

\begin{proof} \ 
\begin{enumerate}[fullwidth, itemsep=1em]

\item Consider paths that start with a down step $D$ and
	end with an up step $U$ with $k-1$ peaks $UD$. Each one will have $k$ conjugates of
	this form because the starting point will become a peak when we take a conjugate. So taking peaks 
	as special vertices we see by Theorem 4 that the number of peaks on or below the $x$-axis is 
	equidistributed.

	We can write such a path as $D^{j_0} U^{i_1} D^{j_1} \cdots U^{i_{k-1}} D^{j_{k-1}}
	U^{i_k}$ where \[i_1 + i_2 + \cdots + i_k = n+1; \hspace{.5cm} i_l > 0 \] and 
	\[j_0 + j_1 + \cdots + j_{k-1} = n; \hspace{.5cm}  j_l > 0\]
	for $l = 0, \dots, k$. The number of solutions of these equations is $\binom{n-1}{k-1} \binom{n}{k-1}$. 
	Since each path has $k$ conjugates of this form, the number of paths
	with $j$ peaks on or below the $x$-axis is given by $\tfrac{1}{k} \binom{n-1}{k-1}
	\binom{n}{k-1}$.

\item Since the peaks and the valleys are interchangable, by replacing the up steps with down steps the proof of the the second part is exactly the same as the first part.


\item Consider paths that start with an up step $U$ and
	end with an up step $U$. We know that the number of peaks plus the number of double rises is
	equal to $n$. So if we consider paths with $k$ $UD$s then each
	path will have $n-k$ double rises. So there will be $n-k+1$ conjugates that start and end with an up step.  We can write such a path as $U^{i_1} D^{j_1} \cdots U^{i_k} D^{j_k} U^{i_{k+1}}$ where
	\[i_1 + i_2 +\cdots + i_{k+1} = n+1 ; \hspace{.5cm} i_l > 0, \textnormal{ for } l=1, \dots, k+1\] 
	and \[j_1 + j_2 +\cdots + j_k = n; \hspace{.5cm} j_m > 0, \textnormal{ for } m=1,  \dots, k.\]
	The number of solutions of these equations is $\binom{n}{k} \binom{n-1}{k-1}$.
	Since there are $n-k+1$ conjugates of this form, the number of paths with $j$ double rises on or below the $x$-axis is given by $\tfrac{1}{n-k+1} \binom{n}{k}
	\binom{n-1}{k-1}$.

\item Consider paths that start with a down step $D$ and
	end with a down step $D$. We know that the number of valleys plus the number of double falls is
	equal to $n-1$. So if we consider paths with $k$ $DU$s then each
	path will have $n-k-1$ double falls. So we can
	write such a path as $D^{i_1} U^{j_1} \cdots
 	U^{j_k} D^{i_{k+1}}$ where
	\[i_1 + i_2 +\cdots + i_{k+1} = n ; \hspace{.5cm} i_l > 0 \textnormal{ for } l=1, 2, \dots, k+1\] and 
	\[j_1 + j_2 +\cdots + j_k = n+1; \hspace{.5cm} j_m> 0 \textnormal{ for } m=1, \dots, k.\]
	The number of solutions of these equations is $\binom{n-1}{k} \binom{n}{k-1}$.
	Since there are $n-k$ conjugates of this form, the number of paths with $j$ double falls on or below the $x$-axis is given by $\tfrac{1}{n-k} \binom{n-1}{k} \binom{n}{k-1}$.

\item If we consider paths that start with an up step $U$ with $k$ peaks $UD$ and we do not care how 
	they end then we get $n+1$ conjugates of 
	this form. We can write such a path as $U^{i_1} D^{j_1} U^{i_2} \cdots U^{i_k}
	D^{j_k} U^{i_{k+1}-1}$ where \[i_1 + i_2 + \cdots + i_{k+1} - 1 = n+1; \hspace{.5cm} i_l > 0 \] for $l = 1, \dots, k+1$ and 
	\[j_1 + j_2 + \cdots + j_k = n; \hspace{.5cm}  j_l > 0\]
	for $l = 1, \dots, k$. The number of solutions of these equations is $\binom{n+1}{k} \binom{n-1}{k-1}$. 
	Since each path has $n+1$ conjugates of this form, the number of paths
	with $j$ up steps on or below the $x$-axis is given by $\tfrac{1}{n+1} \binom{n+1}{k} \binom{n-1}{k-1}$.

\item If we consider paths that start with a down step $D$ with $k$ valleys $DU$, each one will have $n$
	 conjugates of this form. We can write such a path as $D^{i_1} U^{j_1} \cdots D^{i_k}
	U^{j_k}  D^{i_{k+1} - 1}$ where \[i_1 + i_2 + \cdots + i_{k+1} - 1 = n; \hspace{.5cm} i_l > 0 \]
	for $l = 1, \dots, k$ and 
	\[j_1 + j_2 + \cdots + j_k = n+1; \hspace{.5cm}  j_l > 0\]
	for $l = 1, \dots, k+1$. The number of solutions of these equations is $\binom{n}{k} \binom{n}{k-1}$. 
	Since each path has $n$ conjugates of this form, the number of paths
	with $j$ down steps on or below the $x$-axis is given by $\tfrac{1}{n} \binom{n}{k} \binom{n}{k-1}$. \qedhere
\end{enumerate}
\end{proof}

	Notice that from part five of Theorem \ref{t6} we can find an
analogue of the classical Chung-Feller theorem for Narayana numbers in terms of {\it decending runs}. A decending run in a path is a maximal consecutive sequence of down steps. For example, $U\underline{D}U\underline{DD}UUU\underline{DD}$ has $3$ decending runs. If we remove the first up step of the paths as described in part five and shift the paths down one level, we get paths in $\mathcal{P}(n,1,0)$ that start and end on the $x$-axis. If these paths start with an up step they will have $k$ peaks or $k$ decending runs. If they start with a down step then they will have $k-1$ peaks but $k$ decending runs. Therefore the equivalent Narayana-Chung-Feller theorem is 
\begin{theorem}[Narayana-Chung-Feller Theorem]
Among the paths in $\mathcal{P}(n,1,0)$  with $k$ decending runs, the number of paths with $i$ up steps below the $x$-axis is independent of $i$ for $i=0, \dots, n$, and is the Narayana number $\tfrac{1}{n+1} \binom{n+1}{k}  \binom{n-1}{k-1}$.
\end{theorem}
The Narayana number formula  $\tfrac{1}{k-1} \binom{n}{k} \binom{n-1}{k-2}$
did not fit into this picture. But we have a nice combinatorial
interpretation for this form in section \ref{sec:tenth}. 
\clearpage
\section{Circular peaks}\label{sec:sixth}
	We will introduce here the notion of circular peaks to give yet another application of the  generalized Chung-Feller theorem. In addition to the six forms of the Narayana number formula presented in the previous section there is another form given by
\begin{equation}
N(n,k) = \dfrac{1}{2n+1} \left ( \binom{n}{k-1} \binom{n}{k} + \binom{n+1}{k} \binom{n-1}{k-1}  \right ).
\end{equation}
We'll present a theorem in this section that will give a combinatorial interpretation of this form of the Narayana number formula.
\begin{defn}
 	For any path $p \in \mathcal{P}(n, 1 ,1)$ we call every peak a {\it circular peak}. If $p$ starts with a down step 
	and ends with an up step then the initial vertex will also be considered as a circular peak.
\end{defn}
	Note that circular peaks are preserved under arbitrary conjugation. 

\begin{theorem}
	The number of paths in $\mathcal{P}(n,1,1)$ with $k$ circular peaks having $j$ 
	vertices on or below the $x$-axis is independent of $j$ for $j=1, \dots, 2n+1$. The number of such paths is 
	given by the Narayana number $N(n,k) = \tfrac{1}{2n+1} \left ( \binom{n}{k-1} \binom{n}{k} + \binom{n+1}{k} \binom{n-1}{k-1}  \right )$.
\end{theorem}

\begin{proof}
	We consider paths with $n+1$ up steps and $n$ down steps with $k$ circular peaks. To find the total number 
	of paths we need to consider two cases.

{\it Case 1:} Paths starting with a down step. This kind of path has 
	$k-1$ peaks if the path ends with an up step and $k$ peaks if it ends with a down step. The path can be represented 
	by  $D^{i_1}U^{j_1} D^{i_2} U^{j_2}\dots D^{i_k}U^{j_k}D^{i_{k+1}-1}$ where
	\[i_1 + i_2  +\cdots + i_k + i_{k+1} -1 = n ; \hspace{.5cm} i_l > 0\] and 
	\[j_1 + j_2 + \cdots + j_k = n+1; \hspace{.5cm} j_l > 0\]
	for each $l = 1, 2, \dots, k$. The number of solution is $\binom{n}{k} \binom{n}{k-1}$.

{\it Case 2:} Paths starting with an up step. This kind of path has $k$ peaks. The path 
	can be represented by  $U^{i_1} D^{j_1} U^{i_2}\dots U^{i_k} D^{j_k}U^{i_{k+1}-1}$ where
	\[i_1 + i_2  +\cdots + i_{k+1} = n+2 ; \hspace{.5cm} i_l > 0 \textnormal{ for each } l=1, 2, \dots, k+1\]
	 and	\[j_1 + j_2 +\cdots + j_k = n; \hspace{.5cm} j_m > 0 \textnormal{ for each } m=1, 2, \dots, k\]
	The number of solution is $\binom{n+1}{k} \binom{n-1}{k-1}$.
	

	Adding the two we get
\begin{equation}
\label{t7}
\begin{aligned} 
   	 \binom{n}{k-1} \binom{n}{k} + \binom{n+1}{k} \binom{n-1}{k-1} 
   	& = (2n+1) \dfrac{1}{n+1} \binom{n+1}{k} \binom{n-1}{k-1} \\
   	& = (2n+1) N(n,k).
\end{aligned}
\end{equation}

	We know that circular peaks are preserved under conjugation and there are $2n + 1$ conjugates of these 
	paths. So using Theorem \ref{t3} dividing \eqref{t7} by $2n + 1$ we see that the number of paths with $j$ 
	vertices on or below the $x$-axis is given by the Narayana number $ \tfrac{1}{2n+1} \left ( \binom{n}{k-1} \binom{n}{k} + \binom{n+1}{k} \binom{n-1}{k-1}  \right )$.
\end{proof}	

\chapter{Other number formulas}
\section{Motzkin, Schr\"oder, and Riordan number formulas}
	In this section we'll consider paths having different types of steps, in particular Motzkin and Schr\"oder paths.
	We'll see that the generalized Chung-Feller theorem can also be applied to the Motzkin and Schr\"oder number formulas.
	
\begin{defn}
	Let us define 
	$\mathcal{Q}(k, l, r, s, h)$ to be the set of paths having the step set $M = \{(1,1), (s, 0), (1, -r)\}$ 
	that lie in the half plane $\mathbb{Z}_{\ge 0}\times \mathbb{Z}$ ending at $((r+1)k + s l + h, h)$ with $r k+h$ up steps, $k$ down steps, and $l$ flat steps.
	The paths in $\mathcal{Q}(k, l, 1, 1, 0)$ that lie in the quarter plane $\mathbb{Z}_{\ge 0}\times \mathbb{Z}_{\ge 0}$ are 
	known as {\it Motzkin paths} and the paths in $\mathcal{Q}(k, l, 1, 2, 0)$ that lie in the quarter plane $\mathbb{Z}_{\ge 0}\times \mathbb{Z}_{\ge 0}$ are known as {\it Schr\"oder paths}. In this section we'll only consider $s$ having the value $1$ or $2$. 
\end{defn}

All the paths discussed before including the Motzkin paths have steps of unit length. Therefore the total number of steps of the paths coincided with the length of the path. But from now we'll define the length of the path to be the $x$-coordinate of the endpoint. So paths in $\mathcal{Q}(k, l, r, s, h)$ have length $(r+1)k + s l + h$ and total number of steps $(r+1)k + l + h$. With this definition we can see that the difference between the Schr\"oder paths and the Motzkin paths is due to the length of the horizontal steps. The horizontal steps of the Schr\"oder paths are of length two. So the Schr\"oder paths that start and end on the $x$-axis have even length. 

Let us define	
\begin{equation} \label{ms}
T(k,l) = \dfrac{1}{k+1}\binom{2 k+ l}{2 k} \binom{2 k}{k} = \binom{2k+l}{2k} C_k.
\end{equation}
Then $T(k,l)$ counts paths in $\mathcal{Q}(k, l, 1, s, 0)$ because the number of ways to place the flat steps is $\binom{2 k+ l}{2 k}$ and after placing the flat steps we can place in $C_k$ ways the up and down steps. Replacing $2k+l$ by $n$ or $k+l$ by $n$ in \eqref{ms} we get the following two formulas. 
\begin{align}\label{mtz}
 M(n,k) &= \dfrac{1}{k+1} \binom{n}{2 k} \binom{2k}{k}\\ \label{sch}
 R(n,k) &=   \dfrac{1}{k+1} \binom{n+k}{2 k} \binom{2k}{k}.
\end{align}
Here $ M(n,k)$ counts Motzkin paths in $\mathcal{Q}(k, n-2k, 1, 1, 0)$ with $k$ up steps, $k$ down steps and $n-2k$ flat steps and the Motzkin number \cite{motz} $M_n = \sum_{k=0}^{\lfloor n/2 \rfloor} M(n,k)$ counts Motzkin paths of length $n$.  The first few Motzkin numbers (sequence A001006 in OEIS) are $1, 1, 2, 4, 9,	21, 51, 127, 323, 835, 2188, 5798, \dots$. Also $R(n,k) $ counts Schr\"oder  paths in $\mathcal{Q}(k, n-k, 1, 2, 0)$ with $k$ up steps, $k$ down steps and $n-k$ flat steps and the Schr\"oder number $R_n = \sum_{k=0}^{n} R(n,k)$ counts Schr\"oder paths of semi-length $ n=  k +  l$. The first few Schr\"oder numbers (sequence A006318 in OEIS) are
$1, 2, 6, 22, 90, 394,\dots$. 

There is a simple relation between \eqref{mtz} and \eqref{sch} given by \[M(n+k,k) = R(n,k). \] 
Below is the table of values of $T(k,l)$ for $k,l = 0, \dots, 6$.
\begin{center}
\begin{tabular}{c|rrrrrrr}
\text{\em k} \textbackslash \! \text{\em l} &0&1&2&3&4&5&6 \\ \hline 
0& 1 & 1 & 2 & 5 & 14 & 42 & 132  \\
1& 1 & 3 & 10 & 35 & 126 & 462 & 1716 \\
2& 1 & 6 & 30 & 140 & 630 & 2772 & 12012 \\
3& 1 & 10 & 70 & 420 & 2310 & 12012 & 60060 \\
4& 1 & 15 & 140 & 1050 & 6930 & 42042 & 240240 \\
5& 1 & 21 & 252 & 2310 & 18018 & 126126 & 816816 \\
6& 1 & 28 & 420 & 4620 & 42042 & 336336 & 2450448
\end{tabular}
\end{center}
\ \\
It is interesting to see that we can write $T(k,l)$ in the following seven forms,
\begin{eqnarray*}
T(k,l)  &=&  \dfrac{1}{k+1} \binom{2k+l}{2 k} \binom{2k}{k}
		=\dfrac{1}{k} \binom{2k+l}{2 k} \binom{2k}{k-1}\\
	&=&  \dfrac{1}{k+l+1} \binom{2k+l}{k} \binom{k+l+1}{k+1}	
		=  \dfrac{1}{k+l} \binom{2k+l}{k+1} \binom{k+l}{k}\\
		&=&  \dfrac{1}{2k+1} \binom{2k+l}{2k} \binom{2k+1}{k} 
		=  \dfrac{1}{l} \binom{2k+l}{k} \binom{k+l}{k+1}\\
		&=& \dfrac{1}{2k+l+1} \binom{2k+l+1}{2k+1} \binom{2k+1}{k}.
\end{eqnarray*}
Note that when $l=0$ these formulas reduce to the three forms of the Catalan numbers except for the one with $\tfrac{1}{l}$ in front. Similar to the Catalan and the Narayana number formulas, we will give a combinatorial interpretation of the different formulas for $T(k,l)$ in the following theorem. 
\begin{theorem}\label{t8} \ 
\begin{enumerate}
\item \label{n1} The number of paths in $\mathcal{Q}(k, l, 1, s, 1)$ that start with an up step with 
	exactly $i$ up steps starting on or below the $x$-axis for $i = 1, 2, \dots, k+1$ is $\tfrac{1}{k+1}\binom{2 k+ l}{2k} \binom{2 k}{k}$.

\item The number of paths in $\mathcal{Q}(k, l, 1, s, 1)$ that start with a down step with 
	exactly $i$ down steps starting on or below the $x$-axis for $i = 1, 2, \dots, k$ is $\tfrac{1}{k}\binom{2 k+ l}{2k} \binom{2k}{k-1}$.

\item The number of paths in $\mathcal{Q}(k, l, 1, s, 1)$ that start with a flat step with 
	exactly $i$ flat steps starting on or below the $x$-axis for $i = 1, 2, \dots, l$ is $\tfrac{1}{l}\binom{2 k+ l}{k} \binom{k+l}{k+1}$.
	
\item The number of paths in $\mathcal{Q}(k, l, 1, s, 1)$ that start with an up step or flat step with 
	exactly $i$ up or flat steps starting on or below the $x$-axis for $i = 1, 2, \dots,  k+ l+1$ is $\tfrac{1}{k+ l+1}\binom{2 k+ l}{k} \binom{ k+ l+1}{k+1}$.

\item The number of paths in $\mathcal{Q}(k, l, 1, s, 1)$ that start with a down step or a flat step with 
	exactly $i$ down or flat steps starting on or below the $x$-axis for $i =1, 2, \dots,  k+ l$ is $\tfrac{1}{k+ l}\binom{2 k+ l}{k+1} \binom{ k+ l}{k}$.

\item The number of paths in $\mathcal{Q}(k, l, 1, s, 1)$ that start with an up or a down step with 
	exactly $i$ up or down steps starting on or below the $x$-axis for $i = 1, 2, \dots, 2k+1$ is $\tfrac{1}{2k+1}\binom{2 k+ l}{2k} \binom{2k+1}{k}$.
	
\item The number of paths in $\mathcal{Q}(k, l, 1, s, 1)$ with exactly $i$ vertices on or below the $x$-axis for $i = 1, 2, \dots, 2k+l+1$ is $\tfrac{1}{2 k+ l+1}\binom{2 k+ l+1}{2k+1} \binom{2k+1}{k}$.
\end{enumerate}
\end{theorem}
\begin{proof}
The proof is straightforward using similar arguments to those in the proof of Theorem \ref{t5}. For example the paths in Theorem \ref{t8}(\ref{n1}) that start with an up step have a total of $2k+l+1$ steps. Since the paths start with an up step we can choose $2k$ places from the remaining $2k+l$ places in $\binom{2k+l}{2k}$ ways for the up and down steps and then choose $k$ places from the $2k$ chosen places in $\binom{2k}{k}$ ways to place the down steps. Since there are $k + 1$ conjugates for each path that start with an up step, the number of paths with exactly $i$ up steps on or below the $x$-axis for $i = 1, 2, \dots, k+1$ is $\tfrac{1}{k+1}\binom{2k+l}{2k} \binom{2k+1}{k}$.
\end{proof}

It is also easy to make the connection between these paths and Motzkin and Schr\"oder paths which end at height $0$ rather than height $1$. For example, consider the paths in Theorem \ref{t8}(1) that start with an up step and end at height one keeping track of the up steps starting on or below the $x$-axis. According to the theorem, the number of these paths with $i$ up steps starting below the $x$-axis is independent of $i$. So if we remove the first up step of these paths and shift the paths down one level then we get paths that start and end on the $x$-axis, and have $i$ up steps starting below the $x$-axis. Furthermore if we consider $i=0$ then all the steps must start on or above the $x$-axis and we get exactly the Motzkin or Schr\"oder paths. On the other hand if we take $i$ as large as possible then removing the first up step and shifting the path down one level gives us the negatives of the Motzkin or Schr\"oder paths.

Next we look at similar relations with the Riordan and small Schr\"oder numbers. The number of Motzkin paths of length $n$ with no horizontal steps at level $0$ are called Riordan numbers (sequence A005043 in OEIS) and the number of Schr\"oder paths of length $n$ with no horizontal steps at level $0$ are called small Schr\"oder numbers (sequence A001003 in OEIS). Therefore Riordan and small Schr\"oder paths are Motzkin and Schr\"oder paths respectively without any flat steps on the $x$-axis. It can be shown that 
\begin{equation} \label{rs}
Z(k,l) = \dfrac{1}{k}\binom{2 k+ l}{ k-1} \binom{k+l-1}{k-1}
\end{equation}
counts paths in $\mathcal{Q}(k, l, 1, s, 0)$ with no flat step on the $x$-axis. 

Replacing $2k+l$ by $n$ or $k+l$ by $n$ in \eqref{rs} we get the following two formulas. 
\begin{eqnarray}\label{rio}
 J(n,k)&=& \dfrac{1}{k} \binom{n-k-1}{k-1} \binom{n}{k-1}\\ \label{lsch}
 S(n,k) &=&   \dfrac{1}{k} \binom{n-1}{k-1} \binom{n+k}{k-1}.
\end{eqnarray}
Here $ J(n,k)$ counts Riordan paths in $\mathcal{Q}(k, n-2k, 1, 1, 0)$ with $k$ up steps, $k$ down steps and $n-2k$ flat steps and the Riordan number $J_n = \sum_{k=0}^{\lfloor n/2 \rfloor} J(n,k)$ counts Riordan paths of length $n$.  The first few Riordan numbers $0, 1, 1, 3, 6, 15,	36, 91, \dots$. On the other hand $S(n,k)$ counts small Schr\"oder  paths in $\mathcal{Q}(k, n-k, 1, 2, 0)$ with $k$ up steps, $k$ down steps, and $n-k$ flat steps and the small Schr\"oder number $S_n = \sum_{k=0}^{n} S(n,k)$ counts small Schr\"oder paths of semi-length $ n=  k +  l$. The first few small Schr\"oder numbers are $1, 1, 3, 11, 45, 197, \dots$. 

The relation between \eqref{rio} and \eqref{lsch} is similar to the relation between the Motzkin and large Schr\"oder number formulas, \[J(n+k,k) = S(n,k). \] 
The following table illustrates $Z(k,l)$ for values of $l$ and $k$ from $1$ to $6$.

\begin{center}
\begin{tabular}{c|rrrrrrr}
\text{\em k} \textbackslash \! \text{\em l} &1&2&3&4&5&6 \\ \hline 
 1&1 & 2 & 5 & 14 & 42 & 132 \\
 2&1 & 5 & 21 & 84 & 330 & 1287 \\
 3&1 & 9 & 56 & 300 & 1485 & 7007 \\
 4&1 & 14 & 120 & 825 & 5005 & 28028 \\
 5&1 & 20 & 225 & 1925 & 14014 & 91728 \\
 6&1 & 27 & 385 & 4004 & 34398 & 259896 \\
\end{tabular}
\end{center}

There are several forms of $Z(k,l)$ as well. More precisely five, as follows 
\begin{eqnarray*}
Z(k,l)  &=&  \dfrac{1}{k+l} \binom{k+l}{k} \binom{2k+l}{k-1}
		=\dfrac{1}{k} \binom{k+l-1}{k-1} \binom{2k+l}{k-1}\\
	&=&  \dfrac{1}{k+l+1} \binom{2k+l}{k} \binom{k+l-1}{k-1}	
		=  \dfrac{1}{l} \binom{2k+l}{k-1} \binom{k+l-1}{k}\\
		&=& \dfrac{1}{2k+l+1} \binom{2k+l+1}{k} \binom{k+l-1}{k-1}.
\end{eqnarray*}

Although these forms suggest that there may exist a nice combinatorial interpretation like Theorem \ref{t8}, we do not have one so far. 

The relation between these formulas can be viewed nicely using the following diagram which also shows the relation between Motzkin and Riordan number formulas and large and small Schr\"oder number formulas.
\begin{diagram}
 &     &        &T(k,l)  &           &  \\
 &     &\ldLine^{n=2k+l} &         & \rdLine^{n=k+l}   &\\
 &M_n=\sum_k M(n,k)&&\rLine^{M(n+k,k)=R(n,k)}  &           &R_n=\sum_k R(n,k)\\
  & &        &        &           & \\
 &\dLine~{M_n=J_n+J_{n+1}} &        &        &           &\dLine~{R_n=2 S_n}  \\
 & &        &        &           & \\
 &J_n=\sum_k J(n,k)&&\rLine^{J(n+k,k)=S(n,k)}  &           &S_n=\sum_k S(n,k)\\ 
   &   &\rdLine_{n=2k+l} &         & \ldLine_{n=k+l}  & \\
 &     &        &Z(k,l).
\end{diagram}
Moreover there is a simple relation between \eqref{ms} and \eqref{rs} given by 
\begin{equation} \notag
T(k,l) = Z(k+1,l) +Z(k,l+1).
\end{equation}
\clearpage
\section{A combinatorial proof of the relation between large and small Schr\"oder numbers and between Motzkin and Riordan numbers}\label{sec:seventh}

It is well known \cite{dd} that 
\begin{equation} \label{lssch}
R_n = 2 S_n
\end{equation}
for $n \ge 1$. Shapiro and Sulanke \cite{sch1}, Sulanke \cite{sch2} and Deutsch \cite{sch3} have given  bijective proofs of \eqref{lssch}. In \cite{sch3} Deutsch uses the notion of short bush and tall bush (rooted short bush) with $n+1$ leaves to show his bijection. 

The small Schr\"oder number $S_n$ for $n \ge 1$ is the number of Schr\"oder paths with no flat steps on the $x$-axis. Marcelo Aguiar and Walter Moreira \cite{ma} noted that the Schr\"oder paths counted by the large Schr\"oder numbers $R_n$ fall in two classes, those with flat steps on the $x$-axis, and those without and the number of paths in each class is the small Schr\"oder number $S_n$.

This is quite easy to see. Consider a Schr\"oder path with at least one flat step on the $x$-axis. Now we remove the last flat step that lies on the $x$-axis and elevate the path before the flat step by adding an up step at the begining and a down step at the end. The resulting path will have no flat step on the $x$-axis. To go back consider a nonempty Schr\"oder path with no flat step on the $x$-axis. This kind of path must start with an up step. So we look at the part of the path that returns to the $x$-axis for the first time. We remove the up and down step from the two ends of this part and replace them with a flat step after this part. The resulting path is a Schr\"oder path with at least one flat step on the $x$-axis. 

From Theorem \ref{t8}(4) we find another combinatorial proof of \eqref{lssch}. 
Consider the paths described in Theorem \ref{t8}(4) that start with a flat or a down step and end at height one of length $2 n +1 (= 2k +2l +1) $. Among these paths consider those with all the flat or down steps on or below the $x$-axis. These are counted by the large Schr\"oder numbers. The following figure illustrates a path of this form of length $25$.
\begin{figure}[h!]
\begin{center}
\includegraphics[width=0.80\textwidth]{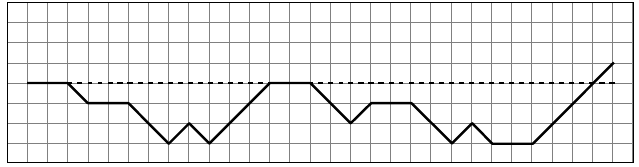} 
	\caption{A path in $\mathcal{Q}(9, 5, 1, 2, 1)$ with all flat or down steps on or below the $x$-axis.}
	\end{center}
\end{figure}
Removing the last up step of these paths gives us the negative Schr\"oder paths. According to the theorem these are equinumerous with those with exactly one flat or down step on or below the $x$-axis. But these fall into two classes, those starting with a flat step and those starting with a down step.

Let $p$ be a path of this form. If $p$ starts with a flat step then it cannot have any other flat step on the $x$-axis, but it may touch the $x$-axis. Moreover the rest of the path cannot go below the $x$-axis. So if we remove the first flat step and add a down step at the end of $p$ we get a Schr\"oder path that does not have a flat step on the $x$-axis. Also exchanging a flat step with a down step reduces the length of the path to $2n$. These paths are counted by the small Schr\"oder numbers $S_n$ \eqref{lsch}.   

On the other hand if $p$ starts with a down step then it must have an up step immediately after that and the rest of the path cannot have any flat step on the $x$-axis and must lie above the $x$-axis, although it may touch the $x$-axis. So if we remove the initial two steps (DU) from $p$ and add a down step at the end we again get a Schr\"oder path of length $2n$ that does not have a flat step on the $x$-axis. Adding these two cases we get the large Schr\"oder numbers $R_n$. This shows that $R_n = 2 S_n$.

We can also look at similar relations between Motzkin and Riordan numbers. We know that the Motzkin and the Riordan numbers are related by the relation 
\begin{equation} \label{mrnum}
M_n = J_n + J_{n+1}.
\end{equation} 
Here we can use the same argument that we used for Schr\"oder numbers to give a combinatorial interpretation.
 
Consider the paths described in Theorem \ref{t8}(4) that start with a flat or an up step with length $n + 1 (= 2k + l +1)$ and end at height one. Among these paths consider those with all the flat or down steps on or below the $x$-axis. Since all the steps of these paths except the last stay on or below the $x$-axis, removing the last up step gives us the negatives of the Motzkin paths. These are counted by the Motzkin numbers $M_n$ and these are equinumerous, by Theorem \ref{t8}(4), with those paths with exactly one flat or up step on or below the $x$-axis. 

But these also fall into two classes, those starting with a flat step and those starting with a down step. Let $q$ be a path of this form. If $q$ starts with a flat step then it cannot have any other flat step on the $x$-axis. Moreover the rest of the path will lie above the $x$-axis although it may touch the $x$-axis. So if we remove the first flat step and add a down step at the end we get a Motzkin path of length $n+1$ that does not have a flat step on the $x$-axis. Since exchanging the flat step with a down step does not change the length of the path, these paths are counted by the Riordan numbers $J_{n+1}$.   

On the other hand if $q$ starts with a down step then it must have an up step immediately after that and the rest of the path must lie above the $x$-axis. So if we remove the initial two steps (DU) from $q$ and add a down step at the end we get a Motzkin path of length $n$ that does not have a flat step on the $x$-axis and these are counted by the Riordan numbers $J_n$. This shows the relation \eqref{mrnum}.

\chapter{Generating functions}
Generating functions are very useful in lattice path enumeration. Finding generating functions is equivalent to finding explicit formulas. Generating functions can be applied in many different ways, but the simplest is the derivation of functional equations from combinatorial decompositions. For example, every Dyck path can be decomposed into ``prime" Dyck paths by cutting it at each return to the $x$-axis:
\vspace{-4pt}
\begin{figure}[ht!]
	\centering
		\includegraphics[width=0.70\textwidth]{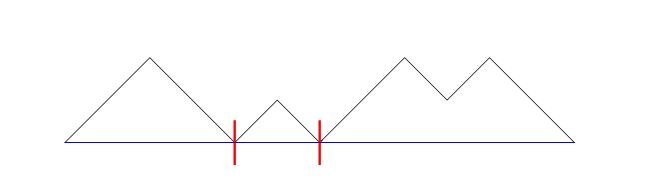}
	\caption{Primes}
	\label{fig:Prime decomposition}
\end{figure}

Moreover, a prime Dyck path consists of an up step, followed by an arbitrary Dyck
path, followed by a down step. It follows that if $c(x)$ is the generating function for Dyck
paths (i.e., the coefficient of $x^n$ in $c(x)$ is the number of Dyck paths with $2n$ steps) then
$c(x)$ satisfies the equation $c(x) = 1 / (1 - x c(x))$ which can be solved to give the generating
function for the Catalan numbers, 
\begin{equation} \notag
c(x) = \dfrac{1 - \sqrt{1 - 4x}}{2x} = \sum_{n=0}^{\infty} \dfrac{1}{n + 1} \binom{2n}{n} x^n.
\end{equation}

Many other lattice path results can be proved by similar decompositions. We'll use mainly three types of decompositions to prove generalized Chung-Feller theorems. The most common form of decomposition is decomposing the path into arbitrary positive and negative primes that start and end on the $x$-axis. We can also consider primes that start and end at height $1$.

For example, let us consider paths in $\mathcal{P}(n,1,0)$. There are $\binom{2n}{n}$ paths in $\mathcal{P}(n,1,0)$ and we know that the generating function for these paths is $\tfrac{1}{\sqrt{1 - 4x}}$. There are several ways we can decompose these paths. First we decompose a path $p$ in $\mathcal{P}(n,1,0)$ into positive and negative primes. The generating function for the positive primes is $xc(x)$ and the generating function for the negative primes is the same. So the generating function for all of these paths is 
\[\dfrac{1}{1-2xc(x)} = \dfrac{1}{\sqrt{1 - 4x}}.\]

Second we can decompose a path $p$ into positive primes separated by (possibly empty) negative paths. 
\begin{figure}[ht!]
	\centering
		\includegraphics[width=0.70\textwidth]{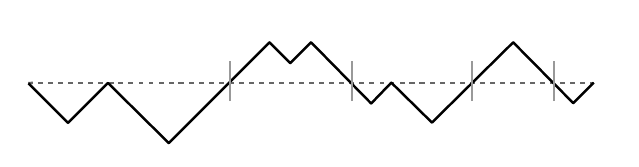}
	\caption{Decomposition of a path into positive primes and negative paths}
	\label{fig:Prime decomposition2}
\end{figure}
Here we have alternating negative paths and positive primes, starting and ending with a negative path. The generating function for negative paths is $c(x)$. So the generating function for all such paths is 
\[\sum_{k=0}^{\infty} c(x) [x c(x) \cdot c(x)]^k = \dfrac{c(x)}{1-x c(x)^2} = \dfrac{1}{\sqrt{1 - 4x}}.\]
 
Finally, we can decompose a path $p$ into alternating positive and negative paths. Let the generating function for nonempty positive and negative paths be $P$ and $N$ respectively. So \[P=N=c(x)-1=x c(x)^2.\]
Therefore the generating function for all paths is 
\begin{equation}
\begin{aligned}
(1+P) \dfrac{1}{1-N P} (1+N) &=  \dfrac{c(x)^2}{1-(x c(x)^2)^2} = \dfrac{c(x)}{1-x c(x)^2}\\ &= \dfrac{1}{\sqrt{1 - 4x}}.
\end{aligned}
\end{equation}
We can also use similar decompositions for paths having different types of steps or ending at other height.
 
\clearpage
\section{Counting with the Catalan generating function}\label{sec:eighth}
In this section we'll give another proof of Theorem \ref{t5} using the generating function approach. First we define the generating functions for the paths described in Theorem \ref{t5}.

Let $x f(x,y)$ denote the generating function for the paths in $\mathcal{P}(n,1,1)$ that start with an up step, where we put a weight of $x$ on the up steps that start on or below the $x$-axis and we put a weight of $y$ on the up steps that start above the $x$-axis. Similarly we denote by $x g(x,y)$ the generating function for the paths in $\mathcal{P}(n,1,1)$ that start with a down step, where we put a weight of $x$ on the the down steps that start on or below the $x$-axis and we put a weight of $y$ on the down steps that start above the $x$-axis and finally we denote by $y h(x,y)$ the generating function for the paths in $\mathcal{P}(n,1,1)$ putting a weight of $x$ on the vertices that are on or below the $x$-axis and a weight of $y$ on the vertices that are above the $x$-axis except for the first vertex.

With these weights, we have the following theorem which is equivalent to Theorem \ref{t5}.
\begin{theorem} \label{e13}
The generating functions $f(x,y)$, $g(x,y)$ and $h(x,y)$ satisfies 
\begin{enumerate} 
\item  $f(x,y) = \sum_{n=0}^{\infty} C_n \sum_{i=0}^{n} x^i y^{n-i} $
\item	 $g(x,y) = \sum_{n=0}^{\infty} C_{n+1} \sum_{i=0}^{n} x^{i} y^{n-i}$
\item	$ h(x,y) = \sum_{n=0}^{\infty} C_n \sum_{i=0}^{2n} x^i y^{2n-i}$
\end{enumerate}
\end{theorem}

\begin{proof} \ 
\begin{enumerate}[fullwidth, itemsep=1em]
\item To prove Theorem \ref{e13}(1) we first show that 
\begin{equation} \notag
	 f(x,y)=\dfrac{1}{1-x c(x) - y c(y)}.
\end{equation}
	Consider paths in $\mathcal{P}(n,1,1)$ starting with an up 
	step and ending at height $1$. We want to count all such paths 
	according to the number of up steps that start on or below the 
	$x$-axis with weights $x$ and $y$ as described above.
	
	Any path $p$ of this form has a total of $2n + 1$ steps with $n+1$ up steps 
	and $n$ down steps. If we remove the first step of $p$ and shift the path one
	level down we get a path in $\mathcal{P}(n,1,0)$ of length $2n$,
	where the up steps originally starting on or below the $x$-axis are now up
    steps starting below the $x$-axis. The generating function for these paths 
    is $f(x,y)$, where every up step below the $x$-axis is weighted $x$ and every 
    up step above the $x$-axis is weighted $y$. 
    	
   	We can factor this path into positive and negative primes, where a positive 
   	prime path is a path in $\mathcal{P}(n,1,0,+)$ that starts with an up step and comes 
	back to the $x$-axis only at the end and a negative prime path is a path
	in $\mathcal{P}(n,1,0,-)$ that starts with a down step and returns to the $x$-axis only at 
	the end. We know that the number of positive prime paths of length $2n$ is the 
	$(n-1)$th Catalan number. So the generating function for the positive prime paths (denoted by $f_1^{+}(x)$) 
	is given by
\begin{equation} \notag
	f_1^{+}(y) = \sum_{n=1}^{\infty} C_{n-1} y^n = y c(y).
\end{equation}
    	Similarly the generating function for the negative prime paths
    	(denoted by $f_1^-(y)$) is given by
\begin{equation} \notag
	f_1^{-}(x) = \sum_{n=1}^{\infty} C_{n-1} x^n = x c(x).
\end{equation}
    	Since an arbitrary path can be factored into $l$
    	primes (positive or negative) for some $l$, the generating
    	function for all paths is         
\begin{equation} \notag
	f(x,y)=\sum_{l=0}^{\infty} (f_1^+ + f_1^-) ^l=\dfrac{1}{1-f_1^+-f_1^-}=\dfrac{1}{1-x c(x) - y c(y)}.
\end{equation}
    	Including the initial up step we get the generating function of paths that start with an up step from $(0,0)$  and end at height $1$ as         
\begin{equation} \notag
	x f(x,y)=\dfrac{x}{1-x c(x) - y c(y)}.
\end{equation}
Now we'll show 
\begin{equation} \label{e13a}
	\dfrac{x c(x) - y c(y)}{ x - y} = \dfrac{1}{1-x c(x) - y c(y)} 
\end{equation} 
or
\begin{equation} \notag
	(x c(x) - y c(y))(1 - x c(x) - y c(y)) = x-y.
\end{equation}
Starting with the left-hand side we get
\begin{align} \notag
    	(x c(x) &- y c(y))(1 - x c(x) - y c(y))\\ \notag
            &=  x c(x) - y c(y) - x^2 c(x)^2 + y^2 c(y^2)\\ \notag
            &= x(1+x c(x)^2) - y(1+y c(y)^2) - x^2 c(x)^2 + y^2 c(y^2)\\ \notag
            &= x + x^2 c(x)^2 - y + y^2 c(y)^2 - x^2 c(x)^2 + y^2 c(y^2)\\ \notag
            &= x - y.
    \end{align}

\item To prove Theorem \ref{e13}(2) we first show that 
\begin{equation} \notag
	 g(x,y)=\dfrac{ c(x) c(y)}{1 - x c(x) - y c(y)}.
\end{equation}
	Consider paths in $\mathcal{P}(n,1,1)$ starting with a down step and ending at height $1$. We want 
    	to count all such paths according to the number of down steps
    	that start on or below the $x$-axis. The generating function of these paths is $x g(x,y)$ with weights $x$ and $y$ on the down steps as defined before. 

	Since the paths start with a down step, they start with a negative
    	prime path. So we can write any path starting from $(0,0)$ with a
   	down step and ending at $(2n+1,1)$ in the form       
\begin{equation}\notag
	p = q^- Q q_*^{+}
\end{equation}
    	where $q^-$ is a negative prime path, $Q$ is an arbitrary 
    	path in $\mathcal{P}(n,1,0)$ that starts and ends on the $x$-axis, and
    	$q_*^{+}$ is a path that stays above the $x$-axis and ends 
    	at $(2n+1,1)$. 
    
	So the generating function for positive prime paths is given by 
\begin{equation} \notag
   	g_1^+(y) = y c(y)
\end{equation}
	and the generating function for negative prime paths $q^-$ is given by 
\begin{equation} \notag
	g_1^-(x) = x c(x).
\end{equation}
	If we add an extra down step to $q_*^{+}$ we get a positive prime path. Therefore $q_*^{+}$ has the 	
	generating function $g_*^{+}(y) = c(y)$. We also know $Q$ has the generating function 
	$(1 - g_1^-(x) - g_1^+(y))^{-1}$. Therefore	the generating function for paths of the form $p$ can be written as    
\begin{align} \notag
 x g(x,y) = g_1^-(x) (1- g_1^-(x) - g_1^+(y))^{-1} g_*^{+}(y) =  \dfrac{x c(x) c(y)}{1 - x c(x) - y c(y)}.
\end{align}
Using the identity \eqref{e13a} we find 
 \begin{align} \notag
    	\dfrac{c(x)c(y)}{ 1 - x c(x) - y c(y)} &= c(x)c(y) \dfrac{x c(x) - y c(y)}{ x - y}\\ \notag
            	  &= \dfrac{x c(x)^2 c(y) - y c(y)^2 c(x)}{ x - y}\\ \notag
            	  &= \dfrac{(c(x)-1) c(y) - (c(y)-1) c(x)}{ x - y}\\ \notag
            	  &= \dfrac{c(x)-c(y)}{x-y}. 
    \end{align}

\item Finally to prove Theorem \ref{e13}(3) we first show that 
\begin{equation} \notag
	 h(x,y)=\dfrac{ c(x^2)c(y^2) }{1- x y c(x^2) c(y^2)}.
\end{equation}
 	We consider any path $p \in \mathcal{P}(n,1,1)$ having the following weights on the steps 
	above the $x$-axis and below the $x$-axis: We weight the steps ending at vertices that lie above the 
	$x$-axis by $y$, and we weight steps ending at vertices on or below the $x$-axis by $x$. The generating function for these paths is $y h(x,y)$. We decompose 
	$p$ in the following way:
\begin{equation} \notag
    	p = p_1^{-} p_1^{+}p_2^{-} p_2^{+} \dots p_m^{-} p_m^{+} p_*^{-} p_*^{+}
\end{equation}    
	for $m \ge 0$, where each $p_i^-$ is a negative path, each $p_i^{+}$ is a positive prime path, $p_*^{-}$ is the last negative path, and 	$p_*^{+}$ is the last positive path that leaves $x$-axis for the last time and ends at height $1$. The generating function of the negative 
	paths (denoted by $h_1^-(x)$) is 
\begin{equation} \notag
	h_1^-(x) =  c( x^2).
\end{equation}
The generating function of the positive prime paths (denoted by $h_1^+(x)$) is
\begin{equation}  \notag
    	h_1^+(x) = x y c(y^2)
\end{equation}
    	and the generating function for paths that look like $p_*^{+}$ is 
\begin{equation} \notag
    	h_*^{+}(x) = y c(y^2).
\end{equation}
    	Therefore the generating function for paths of the form $p$ is
\begin{align} \notag
	y h(x,y) = \dfrac{1}{1- h_1^-(x) h_1^+(y)} h_1^-(x) h_*^{+}(y) =   \dfrac{y c(x^2)c(y^2) }{1- x y c(x^2) c(y^2)}.	
\end{align}
Now to complete the proof we need to show that 
\begin{equation} \label{id3}
\dfrac{c(x^2)c(y^2)}{ 1 - x y c(x^2) c(y^2)} \cdot \dfrac{x - y} {x c(x^2) - y c(y^2)} =1.
\end{equation}
Starting with the left hand side we get
\begin{align} \notag
                    \dfrac{c(x^2)c(y^2)}{ 1 - x y c(x^2) c(y^2)} \cdot & \dfrac{x - y} {x c(x^2) - y c(y^2)}\\ \notag
                    &= \dfrac{c(x^2)c(y^2)(x - y)}{x c(x^2) - y c(y^2)-x^2 y {c(x^2)}^2 c(y^2) - x y^2 c(x^2) {c(y^2)}^2}\\ \notag
                    &=  \dfrac{c(x^2)c(y^2)(x - y)}{(1+ y^2 {c(y^2)}^2) x c(x^2) - (1+ x^2 {c(x^2)}^2) y c(y^2)}\\ \notag
                    &=  \dfrac{c(x^2)c(y^2)(x - y)}{x c(y^2) c(x^2) - y c(x^2) c(y^2)}\\            \notag      
                    &= 1. \qedhere
\end{align}
\end{enumerate}
\end{proof}
\clearpage
\section{The left-most highest point}\label{sec:seventeenth}

In this section we'll show another type of equidistribution property with respect to the left-most highest point of paths in $\mathcal{P}(n,1,1)$. We find that the number of paths in $\mathcal{P}(n,1,1)$ with $i$ steps before the left most highest point is independent of $i$.

Given a sequence $f=(a_1, a_2, \dots, a_n) \in \Lambda$ of distinct real numbers with partial sums $s_0=0$, $s_1=a_1$,$\dots$, $s_n=a_1+\cdots+a_n$, where $\Lambda$ is the set of sequences obtained by permuting the elements of $\{a_1, a_2, \dots, a_n\}$, we define the following two numbers:
\begin{align*}
P(f) &= \text{the number of strictly positive terms in the sequence } (s_0, s_1, \dots, s_n)\\
L(f) &= \text{the smallest index } k=(0,1,\dots,n) \text{ with } s_k = \underset{0\le m \le n}{\text{max}} s_m.
\end{align*}
Thus for any permutation of the sequence $f\in \Lambda$, both $P(f)$ and $L(f)$ are natural numbers between $0$ and $n$ and the {\it equivalence principle} of Sparre Andersen \cite{sparre} states that the distribution $P(f)$ and $L(f)$ over the $n!$ permutations of $\Lambda$ are identical. In this section we show that the equivalence principle of Sparre Andersen gives us another Chung-Feller type phenomenon. This was also studied by Foata \cite{foata}, Woan \cite{ee}, and Baxter \cite{baxter}. 

For a lattice path $p\in \mathcal{P}(n,r,h)$ the numbers $P(f)$ and $L(f)$ becomes the number of vertices of the path $p$ that lie on or above the $x$-axis and the position of the left-most highest vertex respectively. First we consider paths in 
$\mathcal{P}(n,1,h)$. Then we claim that every path ending at height $1$ has a unique conjugate whose left-most highest vertex lies at the end. In other words we have the following theorem.

\begin{theorem}
If $p\in \mathcal{P}(n,1,1)$ is a path whose left-most highest vertex lies at the end then the left-most highest vertex of $\sigma^{i}(p)$ lies at position $2n-i$. 
\end{theorem}

\begin{proof} 
Let us consider the generating function approach. Suppose the path ends at height $h$ and the left-most highest vertex $v$ lies at height $k$. We can decompose the path into two parts $a$ and $b$ at the vertex $v$. Part $a$ of the path starts at the origin and end at height $k$ and part $b$ of the path starts at height $k$ and ends at height $k-h$, where $k \ge h$. We weight the steps before the vertex $v$ by $x$ and the steps after the vertex $v$ by $y$. Then the generating function of the part $a$ is $(x c(x^2))^k$ and the generating function of part $b$ is $c(y^2)(y c(y^2))^{k-h}$. Therefore the generating function of the whole path 
(denoted by $P(x,y)$) is 
\begin{equation}
P(x,y)=\sum_{k=1}^{\infty} (x c(x^2))^k c(y^2)(y c(y^2))^{k-h}.
\end{equation}  
Taking $h=1$ we get 
\begin{equation} \label{lh}
\begin{aligned} 
P(x,y) &= \sum_{k=0}^{\infty} (x c(x^2))^k c(y^2)(y c(y^2))^{k-1}\\
		&= x c(x^2) c(y^2) \sum_{k=0}^{\infty} (x c(x^2))^k (y c(y^2))^{k}\\
		&= \dfrac{x c(x^2) c(y^2)}{1-x y c(x^2) c(y^2)}.	
\end{aligned}
\end{equation}  
From equation \eqref{id3} and \eqref{lh} we find that 
\begin{equation} \notag
P(x,y)=x \dfrac{x c(x^2)-y c(y^2)}{x-y}	
\end{equation}  
and the coefficient of $x^{i+1} y^{2n-i}$ in $P(x,y)$ is $\tfrac{1}{2n+1} \binom{2n+1}{n}$.
\end{proof}
\clearpage
\section{Counting with the Narayana generating function}\label{sec:nineth}

Recall that the Narayana numbers are
\begin{equation} \notag
	N(n,k) = \dfrac{1}{n} \binom{n}{k} \binom{n}{k-1}
\end{equation}
 	for $n \ge 1.$ We can get the Catalan numbers from the Narayana
 	numbers by 
\begin{equation} \label{nara}
	\sum_{k =1}^n N(n,k) = C_n. 
\end{equation}
    	We define the Narayana generating function by   
\begin{equation} \label{e32}
	E(x,  s) = \sum_{1\le m\le n} N(n,m) s^{m-1} x^{n}.
\end{equation}
It is known that $E(x, s)$ can be expressed explicitly as
\begin{equation} \label{exs}
	E(x, s) =  \dfrac{1-x-x s-\sqrt{(1-x+x s)^2-4x s}}{2 x s}.
\end{equation}
	Notice that $E(x,1) =  c(x) -1$. We will use several identities satisfied by the generating function $E$ which can be proved by a straightforward computation which we omit. We list them here
\begin{equation} \label{id}
\begin{aligned}
1+ \dfrac{ s E(x, s ) - t E(x, t)}{s-t} 	=1+ \dfrac{ E(x,s)(1+t  E(x,t))}{1-t E(x,s) E(x,t)}
			&= \dfrac{1 + E(x, t)}{1 -  s  E(x, s) E(x, t)}\\
			&= \dfrac{1 +   E(x, s)}{1 - t E(x, s) E(x, t)}
\end{aligned}
\end{equation}
\begin{equation} \label{id2}
\begin{aligned}
\dfrac{ E(x, s ) - E(x, t)}{s-t} &= \dfrac{  (1 +  E(x, s)) E(x, s) E(x, t)}{1- t E(x, s) E(x, t)}\\
					& = \dfrac{  x(1 +  E(x, t)) E(x, s)}{1-x(1+sE(x,s)) - x t (1+ E(x, t))} .
\end{aligned}
\end{equation}

Next we'll give a generating function proof of Theorem \ref{t6} by decomposing the paths into positive and negative parts or into primes. We recall here that by a peak lying on or below the $x$-axis we mean the vertex between the
up step and the down step lying on or below the $x$-axis and similarly for valleys/double rises/double falls.

\begin{proof}[Proof of Theorem \ref{t6}.] \  
\begin{enumerate}[fullwidth, itemsep=1em] 
\item For the first part of Theorem \ref{t6} we want to count paths $p_1 \in \mathcal{P}(n,1,1)$ that start with a down step and end with an up step according to the number of peaks on or below the $x$-axis. We take $\mathit{L}_{\pk}^+(x, s) $ and $\mathit{L}_{\pk}^-(x, t)$ to be the generating function of the nonempty positive paths and the nonempty negative paths in $\mathcal{P}(n,1,0)$ respectively according to peaks. From \eqref{e32} we see that if $x$ weights the semi-length and $s$ weights the number of peaks then $ s E(x,s)$ is the generating function for nonempty Dyck paths according to peaks. Therefore we can express $\mathit{L}_{\pk}^+(x, s) $  in terms of the Narayana generating function $E(x, s)$ as 
\begin{align} \label{pos}
\mathit{L}_{\pk}^+(x, s) &=\sum_n \sum_{\substack{\text{all nonempty }\\ p \in \mathcal{P}(n,1,0,+)}} x^n s^{\pk(p)}   = s E(x, s)
\end{align}
where $n$ is the semi-length of $p$ and $\pk(p)$ is the number of peaks of $p$. If we reflect a Dyck path about the $x$-axis we get a negative path where the peaks become valleys and the valleys become peaks. Since the number of valleys in a Dyck path is one less than the number of peaks, the generating function $\mathit{L}_{\pk}^-(x, t)$ can be expressed in terms of the Narayana generating function as
\begin{align*}
\mathit{L}_{\pk}^-(x, t) &= \sum_n \sum_{\substack{\text{all nonempty }\\ p \in \mathcal{P}(n,1,0,-)}} x^n t^{\pk(p)}
			= \sum_n \sum_{\substack{\text{all nonempty }\\ p \in \mathcal{P}(n,1,0,+)}} x^n t^{v(p)}\\
			&= E(x, t)
\end{align*}
where $v(p)$ is the number of valleys of $p$. 

Since the path $p_1$ starts with a down step, it starts with a negative path. If we remove the last up step then we can write the remaining path $G_p$ in the form
\begin{equation} \notag
	G_p = g_1^-g_1^+g_2^-g_2^+ \cdots g_m^-g_m^+ g_{p*}^- U
\end{equation}
	where each $g_i^-$ is a nonempty negative path and each  $g_i^+$ is a nonempty positive path for $0\le i \le m$ and $g_{p*}^-$ is the last negative path which can be empty.
\begin{figure}[ht!]
	\centering
		\includegraphics[width=0.80 \textwidth] {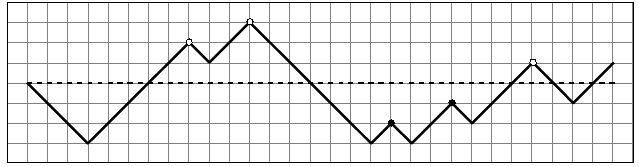}
	\caption{Peaks on or below the $x$-axis}
	\label{fig:latp}
\end{figure}
	Therefore, taking $L_{\pk}(x, s, t)$ to be the generating function for all paths of the form $p_1$ according to the semi-length and number of peaks with weight $s$ on the peaks that lie above the $x$-axis and weight $t$ on the peaks that lie on or below the $x$-axis, we can write
\begin{align*} 
	L_{\pk}(x, s, t) &= \dfrac{1}{1 - \mathit{L}_{\pk}^-(x, t) \mathit{L}_{\pk}^+(x, s)} (1+\mathit{L}_{\pk}^-(x, t))\\ \notag
	 &= \dfrac{1 +  E(x, t)}{1 -  s E(x, t) E(x, s) }\\
	 &= 1 + \dfrac{ s E(x, s ) - t  E(x, t)}{s-t} \quad \quad \text{by \eqref{id}}\\
	&= 1 + \sum_{1\le m \le n} N(n,m) x^n (s^{m-1} + s^{m-2} t+ \cdots + t^{m-1}).
\end{align*}
	This shows that the coefficient of $x^n s^i t^j $ in the expansion of $L_{\pk}(x, s, t)$ is given by the Narayana number $\tfrac{1}{k} \binom{n}{k-1} \binom{n-1}{k-1}$. 

\item The second part of the theorem counts paths $G_v \in \mathcal{P}(n,1,1)$ that start with an up step and end with a down step with respect to the valleys on or below the $x$-axis. For convenience we'll decompose these paths into positive and negative paths with respect to height $1$ instead of the $x$-axis. So a positive/negative path in this case will be a path that starts and ends at height $1$ and stays above/below the $x$-axis respectively.

We take $\mathit{L}_{v}^+(x, s) $ and $\mathit{L}_{v}^-(x, t)$ to be the generating functions of nonempty positive paths and negative paths that start and end at height $1$ according to valleys where $s$ is the weight on valleys that stay above the $x$-axis and $t$ is the weight on valleys that stay on or below the $x$-axis. They may be expressed in terms of the Narayana generating function $E(x, y)$ as follows	
\begin{align*}
	\mathit{L}_{v}^+(x, s) &= \sum_n \sum_{p \in \mathcal{P}(n,1,0,+)} x^n s^{v(p)}
	=  E(x, s)\\
	\mathit{L}_{v}^-(x, t) &= \sum_n  \sum_{p \in \mathcal{P}(n,1,0,-)} x^n t^{v(p)}
			=  t E(x, t).
\end{align*}
 	After the first up step we cut	$G_v$ each time it crosses height $1$. Since the path $G_v$ ends with a down step, it ends with a positive path at height $1$ and $G_v$ will have alternating positive and negative parts after the first up step. So we can write any path $G_v$ in the form
\begin{equation} \notag
	G_v = U g_{v*}^+ g_1^- g_1^+ g_2^- g_2^+ \cdots g_k^-g_k^+ 
\end{equation}
where each $g_i^-$ is a nonempty negative path at height $1$ and each  $g_i^+$ is a nonempty positive path at height $1$ for $0\le i \le k$ and $g_{v*}^+$ is the first positive path at height $1$ that can be empty. The generating function of $g_v^*$ is $1+ \mathit{L}_{v}^+(x, s)$.	
	\begin{figure}[h]
	\centering
		\includegraphics[width=0.80\textwidth]{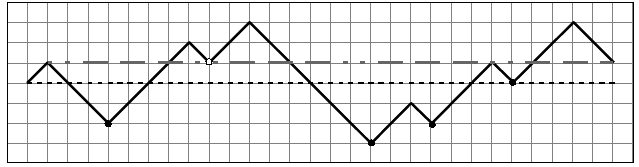}
	\caption{Valleys on or below the $x$-axis}
	\label{fig:latv}
\end{figure}
	If we denote by $L_v(x, s, t)$ the generating function for all such paths $G_v$ according to the semi-length and number of valleys with weight $s$ on the valleys that lie above the $x$-axis and weight $t$ on the valleys that lie on or below the $x$-axis, then we can write
	\begin{align*}
	L_v(x, s, t) &= \dfrac{1}{1-  \mathit{L}_{v}^+(x, s)  \mathit{L}_{v}^-(x, t)} (1+ \mathit{L}_{v}^+(x, s))\\ 
	 &= \dfrac{1+  E(x,s)}{1-  t E(x,s) E(x,t)}\\
	&= 1+ \dfrac{ s E(x, s ) -  t E(x, t)}{s-t} \quad \quad \text{by \eqref{id}}\\ 
	&=1 +  \sum_{1\le m \le n} N(n,m) x^n (s^{m-1} + s^{m-2} t+ \cdots + t^{m-1}).
\end{align*}
	So we see that the coefficient of $x^n s^i t^j $ in the expansion of $L_p(x, s, t)$ is given by the 
	Narayana number $\tfrac{1}{k} \binom{n}{k-1} \binom{n-1}{k-1} $.
	 
\item For the third part of the theorem we would like to count the paths $H_{\dr} \in \mathcal{P}(n,1,1)$ for $n>1$ that start with an up step and end with an up step with respect to the double rises on or below the $x$-axis. Since for each Dyck path the total number of peaks and 
double rises is equal to $n$, it is easy to find the generating function of the positive and negative paths with respect to double rises using \eqref{pos} and the fact that double rises in positive and negative paths have the same distribution. We take $\mathit{L}_{\dr}^+(x, s) $ and $\mathit{L}_{\dr}^-(x, t)$ to be the generating functions of the positive paths and the negative paths in $\mathcal{P}(n,1,0)$ according to double rises. Therefore 
\begin{align*}
 	\mathit{L}_{\dr}^+(x, s) &=\sum_n \sum_{p \in \mathcal{P}(n,1,0,+)}  x^n s^{\dr(p)}
						 = \mathit{L}_{\pk}^+(x s, s^{-1}) =  E(x, s)\\
 	\mathit{L}_{\dr}^-(x, t) &=\sum_n \sum_{p \in \mathcal{P}(n,1,0,-)}  x^n t^{\dr(p)}
						 =  E(x, t)
\end{align*}
	where $\dr(p)$ is the number of double rises of $p$. 
	
	If we decompose $H_{\dr}$ into positive and negative parts we see that whenever the path transitions from negative to positive we get an additional double rise that lies on the $x$-axis and if the last negative part of the path is not empty we get another double rise at the end. So we can write $H_{\dr}$ in the form 
\begin{equation} \notag
	H_{\dr} =  h_b^+ (h_1^- h_1^+ h_2^- h_2^+ \cdots h_k^- h_k^+) h_f^- U
\end{equation}
where each $h_i^-$ is a nonempty negative path and each  $h_i^+$ is a nonempty positive path for $0\le i \le k$, $h_{b}^+$ is the initial nonempty positive part of the path and  $h_{f}^-$ is the last negative path that can be empty.  The generating function of $h_{f}^- $ is $1 + t \mathit{L}_{\dr}^-(x, t) $. 

\begin{figure}[ht]
	\centering
		\includegraphics[width=0.80\textwidth]{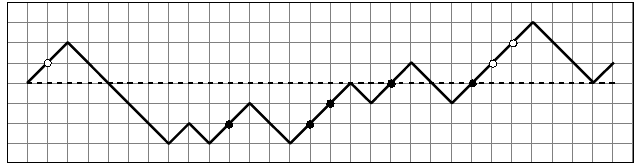}
	\caption{Double-rises on or below the $x$-axis}
	\label{fig:latdr}
\end{figure}
	Therefore taking $L_{\dr}(x, s, t)$ to be the generating function for all such paths according to the semi-length and number of double rises  with weight $s$ on the double rises that lie above the $x$-axis and weight $t$ on the double rises that lie on or below the $x$-axis, we have
\begin{align*} 
	L_{\dr}(x, s, t) &= \mathit{L}_{\dr}^+(x,s) \dfrac{1}{1- \mathit{L}_{\dr}^+(x, s) \mathit{L}_{\dr}^-(x, t) t} (1 + t \mathit{L}_{\dr}^-(x, t)) \\ 
				&=  \dfrac{ E(x,s)(1+t  E(x,t))}{1-t  E(x,s) E(x,t)}\\ 
				&=  \dfrac{ s  E(x, s ) -  t  E(x, t)}{s-t} \quad \quad \text{by \eqref{id}}\\
      				&=  \sum_{1\le m \le n} N(n,m) x^{n} (s^{m-1} + s^{m-2} t+ \cdots + t^{m-1}).
\end{align*}
So we see that the coefficient of $x^n s^i t^j $ in the expansion of $L_{\dr}(x, s, t)$ is given by the 
	Narayana number $\tfrac{1}{n-k+1} \binom{n}{k}  \binom{n-1}{k-1}$.
	
\item The fourth part of the theorem is the same as the third part where paths start and end with a down 	step 
	instead and are counted according to  double falls. Since the number of double rises and the number of double falls in any path have the same distribution they have the same generating function 
\begin{align*}
	\mathit{L}_{\df}^+(x, s) &=\sum_n \sum_{p \in \mathcal{P}(n,1,0,+)}  x^n s^{\df(p)}
						 =  E(x, s)\\
	\mathit{L}_{\df}^-(x, t) &=\sum_n \sum_{T \in \mathcal{P}(n,1,0,-)}  x^n t^{\df(p)}
						 =  E(x, t)
\end{align*}
where $\df(p)$ is the number of double falls of $p$. Similar to part three, note that whenever the path transitions from positive to negative we get an additional double fall that lies on the $x$-axis. These paths have the form 
\begin{equation}
	H_{\df} =  h_{b}^- (h_{1}^+ h_1^- h_2^+ h_2^-\cdots  h_{k}^+ h_{k}^- ) h_{f}^+
\end{equation}
where each $h_i^-$ is a nonempty negative path and each  $h_i^+$ is a nonempty positive path for $0\le i \le k$, $h_{b}^-$ is the initial nonempty negative part of the path and  $h_{f}^+$ is the final positive path that ends at height $1$. The generating function of $h_{f}^+ $ is $(1 + \mathit{L}_{\df}^+(x, t) ) \mathit{L}_{\df}^+(x, t)$ since $h_{f}^+ $ consists of an initial possibly empty positive path followed by an up step followed by a nonempty positive path.

\begin{figure}[ht]
	\centering
		\includegraphics[width=0.80\textwidth]{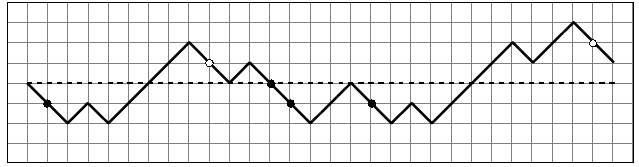}
	\caption{Double-falls on or below the x-axis}
	\label{fig:latdf}
\end{figure}
Therefore taking $L_{\df}(x, s, t)$ to be the generating function for all such paths according to the semi-length and number of double falls with weight $s$ on the double falls that lie above the $x$-axis and weight $t$ on the double falls that lie on or below the $x$-axis, we have
\begin{align*} \notag
	L_{\df}(x, s, t) &= \mathit{L}_{\df}^-(x, t) \dfrac{1}{1- t \mathit{L}_{\df}^+(x, s) \mathit{L}_{\df}^-(x, t) } (1 + \mathit{L}_{\df}^+(x, t) )  \mathit{L}_{\df}^+(x, t)\\
			&= \dfrac{ E(x,t) (1+E(x,s))E(x,s)}{1-t  E(x,t)E(x,s)}\\
			&=  \dfrac{ E(x, s ) -  E(x, t)}{s-t} \quad \quad \text{by \eqref{id2}}\\
			&= \sum_{1< m \le n} N(n,m) x^{n} (s^{m-2} + s^{m-1} t+ \cdots + t^{m-2}).
\end{align*}
So we see that the coefficient of $x^n s^i t^j $ in the expansion of $L_{\df}(x, s, t)$ is given by the 
	Narayana number $\tfrac{1}{n-k} \binom{n}{k-1}  \binom{n-1}{k}$.
\end{enumerate} 
The proofs of the fifth and sixth parts of the theorem are similar, so we leave them to the reader.
\end{proof}
\clearpage	
\section{Up steps in even positions}\label{sec:tenth}
	There is another well-known combinatorial interpretation of the Narayana numbers given by the following theorem. This was one of the first Narayana statistics observed \cite{mm}. We'll give a generalized Chung-Feller theorem that corresponds to this interpretation. D. Callan in \cite{bb} used a similar approach to give a combinatorial interpretation of the formula $\tfrac{j}{n} \binom{kn}{n+j}$.
	
	Let us consider paths in $\mathcal{P}(n-1,1,2)$, i.e., paths that end at height two, according to the number of up steps that start in even positions, where the positions are $0, 1, \dots, 2n-1$. We define an {\it even up step} to be an up step that starts in an even position and an {\it odd up step} to be an up step that starts in an odd position. 

\begin{theorem} \label{th91} \ 
\begin{enumerate}
\item For $k>1$, the number of paths in $\mathcal{P}(n-1,1,2)$ with $k-1$ even down steps that start with a down step with exactly $j$ even down steps on or below the $x$-axis is independent of $j$, $j=1, \dots, k-1$, and is given by the Narayana number $N(n,k) = \tfrac{1}{k-1} \binom{n}{k} \binom{n-1}{k-2}$.

\item The number of paths in $\mathcal{P}(n-1,1,2)$ with $k$ even up steps that start with an up step with exactly $j$ even up steps on or below the $x$-axis is independent of $j$, $j=1, \dots, k$, and is given by the Narayana number $N(n,k) = \tfrac{1}{k} \binom{n-1}{k-1} \binom{n}{k-1}$.
\end{enumerate}
\end{theorem}
\begin{proof}
We'll prove the first part of the theorem and leave the second to the reader as the proof is similar.

  	Any path in $\mathcal{P}(n-1,1,2)$ has $n+1$ up steps and $n-1$ down steps and a total of $2n$ positions ($n$ odd and $n$ even) for the steps. Here we only consider the paths in $\mathcal{P}(n-1,1,2)$ that start with a down step with exactly $k-1$ even down steps. 

Since the paths start with a down step, we have $k-2$ down steps to place in $n-1$ even positions. There are $\binom{n-1}{k-2}$ ways $k-2$ down steps can be even down steps, and the remaining $n-k$ down steps can be assigned to odd positions in $\binom{n}{n-k} = \binom{n}{k}$ ways. So in total there are $\binom{n}{k} \binom{n-1}{k-2}$ paths in $\mathcal{P}(n-1,1,2)$ that start with a down step with $k-1$ even down steps. Any path $p$ of this form in $ \mathcal{P}(n-1,1,2)$ has $k-1$ conjugates that start with even down steps. 

Theorem \ref{spver} only deals with paths that end at height $1$ therefore we cannot apply Theorem \ref{spver} here directly. But we can convert these paths into $2$-colored free Motzkin paths to apply Theorem \ref{spver}. A $2$-colored free Motzkin path is a path with four types of steps,  {\it up}, {\it down}, {\it solid flat} and {\it dashed flat} as shown in the Figure \ref{tcolor}. We can convert paths in $\mathcal{P}(n-1,1,2)$ into $2$-colored free Motzkin paths by taking two steps at a time and converting the $UU$s to $U$, $DD$s to $D$, $UD$s to {\it dashed flat} steps and $DU$s to {\it solid flat} steps. This bijection was given in \cite{tcolor}. 

Since we start with a path that ends at height $2$ the bijection will give us a $2$-colored free Motzkin path that ends at height $1$ with $k-1$ down and solid flat steps. If we take the initial vertices of the down steps and the solid flat steps of the $2$-colored free Motzkin paths as our special vertices then according to Theorem \ref{spver}, out of the $k-1$ conjugates of a $2$-colored free Motzkin path that start with a down or a solid flat step there is only one conjugate having $j$ down or solid flat steps on or below the $x$-axis for $j=1, \dots, k-1$. In terms of a path $p$ in $ \mathcal{P}(n-1,1,2)$ that starts with a down step with $k-1$ even down steps this means that there is one conjugate of $p$ having $j$ even down steps on or below the $x$-axis for $j=1, \dots, k-1$. Therefore the number of paths in $ \mathcal{P}(n-1,1,2)$ that start with a down step having exactly $k-1$ even down steps with $j$ even down steps on or below the $x$-axis is $\tfrac{1}{k-1} \binom{n-1}{k-2} \binom{n}{k}$.
\end{proof}
\begin{figure}[h!]
	\centering
		\includegraphics[width=0.90\textwidth]{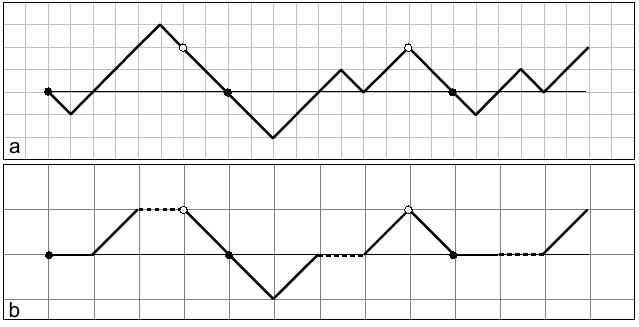}
	\caption{Down steps in even positions: (a) A path in $\mathcal{P}(n-1,1,2)$ and (b) a $2$-colored free Motzkin path of length $9$.}
	\label{tcolor}
\end{figure}
Note that if we take $j=1$ in Theorem \ref{th91}(2) then the path will lie above the $x$-axis except at the beginning. So if we remove the last up step and add a down step at the end we'll get a Dyck path of semi-length $n$. So the number of Dyck paths of semi-length $n$ with $k$ even up steps is $N(n,k)$.

\noindent
{\it Generating function proof:}
We can also prove Theorem \ref{th91} using generating functions: 

	We want to count paths in $\mathcal{P}(n-1,1,2)$ according to even down steps lying on or below the $x$-axis.
We can decompose each path into positive and negative paths. To find the generating function for the 	positive paths we first consider the positive prime paths. We weight the even down steps by 	$s$ and the odd down steps by $t$. A positive prime path does not return to the $x$-axis untill the end. So it starts with an up step followed by a positive path and ends with a down step. Let 
\begin{equation*}
M(x, s, t) = \sum_n \sum_{p \in \mathcal{P}(n,1,0,+)}  s^{e(p)} t^{o(p)} x^{n}
\end{equation*}
	and 
	\begin{equation*}
M^+(x, s, t) = \sum_n \sum_{p \in \mathcal{P}(n,1,0,+) }  s^{e(p)} t^{o(p)} x^{n}
\end{equation*} 
be the generating functions for positive paths and positive prime paths respectively where $n$ is the semi-length, $e(p)$ is the number of even down steps and $o(p)$ is the number of odd down steps. So the positive prime paths have the generating function 	
\begin{equation}
	M^+(x, s, t) = x t M(x, t, s)
\end{equation}
	and the generating function for the positive paths is 
\begin{align} \notag
	M(x, s, t)   &= \dfrac{1}{1 - M^+(x, s, t)} \\ \notag
				&= \dfrac{1}{1 - x t M(x, t, s)} \\
				&= \dfrac{1}{1 - \dfrac{x t \mathstrut}{1- x s M(x, s, t)}}.
\end{align}
	Solving for $M(x, s, t)$ gives us 
\begin{equation} \label{mxst}
	M(x, s, t) = 1+ \dfrac{1 - t x - s x - \sqrt{(1 - t x + s x)^2 - 4 s x}}{2 s x}
\end{equation}
	Note that 
\begin{equation} \label{e35}
\begin{aligned} 
	M(x,1,s) &=1+ E(x,s)\\
	M(x,t,1) &=1+ t E(x,t)
\end{aligned}
\end{equation}
	and 
\begin{equation*}
	E(x,y)=x M(x,y,1) M(x,1,y).
\end{equation*}
We would like to consider the even down steps starting on or below the $x$-axis. So we weight the even down steps starting above the $x$-axis by $a$ and the even down steps starting on or below the $x$-axis by $b$. With these weights the generating functions for the positive primes and negative primes can be written using $M(x, s, t)$ by
\begin{align} \notag
	\pr^{+}(x, a, b) &= x M(x, 1, a)\\ \notag
	\pr^{-}(x, a, b) &= b x M(x, b, 1).
\end{align}
	We can write any path $p \in \mathcal{P}(n-1,1,2)$ that starts with a down step in the form 
	\[p = n_0 (p_1 q_1 \cdots p_{n} q_{n}) p_*,\] 	
	where $n_0$ is the first nonempty negative path, each $p_i$ is a nonempty positive path and each $q_i$ is a nonempty negative path, and $p_*$ is the last positive path that ends at height $2$. These have the generating functions 
\begin{align} \notag
	M^+(x, a, b) &= \dfrac{\pr^{+}(x, a, b)}{1-\pr^{+}(x, a, b)} \\ \notag
	M^-(x, a, b) &= \dfrac{\pr^{-}(x, a, b)}{1-\pr^{-}(x, a, b)} \\ \notag
	M^* (x, a, b) &= x M(x, 1, a) M(x, a, 1)(1+M^+(x, a, b)).
\end{align}
	So we can write the generating function (denoted by $G(x,a,b)$) for the paths $p$ in $\mathcal{P}(n,1,2)$ that start with a down step as
\begin{equation} \notag
\begin{aligned} 
	G(x, a, b) &= M^-(x, a, b) \dfrac{1}{1 -M^+(x, a, b) M^-(x, a, b)} M^*(x, a, b) \\ 			
			&= \dfrac{  b x^2 M(x,b,1) M(x, a,1) M(x,1,a)}{1-x M(x,1,a) - b x M(x,b,1)}
\end{aligned}
\end{equation}
	Using \eqref{e35}  and the identity \eqref{id2} we can write $G(x,a,b)$ as 
\begin{equation} \label{e36}
\begin{aligned} 
	G(x, a, b) &=  \dfrac{b (E(x, a) - E(x, b))}{a - b}\\
			&= b \sum_{1 \le k \le n} N(n+1, k+1) x^{n+1} (a^{k-1} + a^{k-2} b+ \cdots + b^{k-1})\\
		          &= b \sum_{2 \le k \le n-1} N(n, k) x^{n} (a^{k-2} + a^{k-3} b+ \cdots + b^{k-2}).
\end{aligned}
\end{equation}

	Making use of \eqref{e32} and the definition of $E(x, y)$ we find that the coefficient of  $x^{n} a^{k-1-j} b^{j}$, for $j=1, \dots, k-1$ in the expansion of $G(x,a,b)$ in \eqref{e36} is the Narayana number $\tfrac{1}{k-1} \binom{n-1}{k-2} \binom{n}{k}$.

\chapter{Chung-Feller theorems for generalized paths}

In this section we'll consider paths in $\mathcal{P}(n,r,h)$ that go up by one and down by any amount $r>1$ ending at height $h$. We'll use one of the most famous and powerful tools in combinatorics called the Lagrange inversion \cite{qlag, gessel} method to derive some of the generalized formulas. 
\begin{lemma} (Lagrange Inversion) 
Let $g(u)=\sum_{n=0}^{\infty} g_n u^n$, where $g_n$ are indeterminates, and let $f(x)$ be the formal power series in $g_n$ defined by \[f(x)=x g(f(x)).\] Then, for $k>0$,
\begin{equation} \notag
f^k(x) = \sum_{n=1}^{\infty} \dfrac{k}{n} \left[ u^{n-k}\right] g(u)^n.
\end{equation} 
If $\phi(t)$ is a formal Laurent series then another variation of this formula gives
\begin{equation} \notag
\left[ x^{n}\right] \phi(f) = \left[ u^{n}\right] (1- u g'(u)/g(u)) \phi(u) g(u)^n.
\end{equation} 
\end{lemma}
Here $\left[ x^{n}\right] \phi(f)$ denotes the coefficient of 
$x^{n}$ in $\phi(f)$. The above formula is of great importance  in enumeration since many counting problems lead to equations of the form  $f(x)=x g(f(x))$.

\clearpage
\section{Versions of generalized Catalan number formula 1}\label{sec:twelveth}

	Let $f_{h}(x)$ be the generating function for paths in $\mathcal{P}(n,r,h,+)$ that stay strictly above the $x$-axis where each up step has weight $1$ and each down step has weight $x$. 
\begin{figure}[ht]
	\centering
		\includegraphics[width=0.80\textwidth]{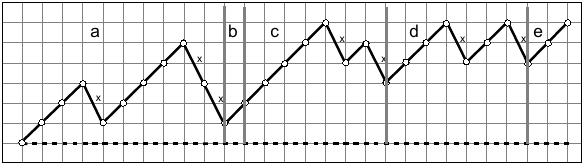}
	\caption{A path in $\mathcal{P}(7,2,6,+)$ decomposed into parts $a, b, c, d, e$}
	\label{fig:latticerk}
\end{figure}
	We can uniquely decompose any path in $\mathcal{P}(n,r,h,+)$ into $h$ consecutive parts where each part is 
	a path in $\mathcal{P}(n,r,1,+)$. As shown in the figure we look at the first part of the path that returns to 
	height $1$ for the last time and remove it. This part (denoted by $a$ in the figure) of the path is in 
	$\mathcal{P}(n,r,1,+)$. Then we remove the next part (part $b$) that ends at height $2$ and so on. This is 
	possible because the rightmost vertex at each level $l \le r$ must be a vertex of the path, i.e., either the endpoint of an up step or the endpoint of a down step. Otherwise it would have to be in the middle of a down 
	step, but then the path ends at height $h \ge l$ so it would have to return later to this height. Therefore the 
	generating function can be written as $f_{h}(x) = f_{1}^{h} (x) = f^{h} (x)$, where $f(x) = f_{1} (x)$ is the 
	generating function for paths in $\mathcal{P}(n,r,1,+)$.

	Now let us consider the paths in $\mathcal{P}(n,r,1,+)$. If a path in $\mathcal{P}(n,r,1,+)$ does not 
	have a down step then it consists of a single up step. Otherwise it ends with a down step and if we remove 
	the last down step we get a path that ends at height $r+1$. So we get the functional equation 
\begin{equation} \label{e6.1}
	f(x) = 1 + x f^{r+1} (x).
\end{equation}
	To solve this by Lagrange inversion we can put in a new redundent variable $z$ and solve the equation \[f(x,z)=z(1+x f^{r+1} (x, z))\] and then set $z=1$ to get 
\begin{equation} \label{e6.2}	
\begin{aligned}	 
	f^{h}(x) &= \sum_{m=h}^{\infty} \frac{h}{m} [t^{m-h}] (1+x t^{r+1})^{m}\\ 
			&= \sum_{m=h}^{\infty} \frac{h}{m} [t^{m-h}] \sum_{n} \binom{m}{n} x^{n} t^{(r+1)n}\\ 
			&= \sum_{n} \frac{h}{(r+1)n+h} \binom{(r+1)n+h}{n} x^{n}, \ \ \textrm{taking} \ m=(r+1)n+h. 
\end{aligned}
\end{equation}
	The coefficients are known as $r$-ballot numbers. In particular we have
\begin{equation} \label{e6.3}
	f(x) = \sum_{n=0}^{\infty} C_n^r x^{n}
\end{equation}
	where we define $C_n^r$ by \[C_n^r =  \dfrac{1}{(r+1)n+1} \binom{(r+1)n+1}{n}.\] Notice that for $r=1$ the 	coefficients reduce to Catalan numbers. The numbers $C_n^r$ are called generalized Catalan numbers or order $r+1$ Fuss-Catalan numbers \cite{fuss, nn}. These numbers were first studied by N. I. Fuss in 1791. They also arise in counting rooted plane 	trees with $r n+1$ leaves in which every non-leaf vertex has exactly $r+1$ children. By ``plane tree", we 	mean that the left-to-right order of children matters. Hilton and Pedersen \cite{oo} observed that $C_{r}^{n}$ counts the subdivisions of a convex polygon into $n$ disjoint $(r+1)$-gons by noncrossing diagonals.  

	These numbers can also be written in three different ways like the Catalan numbers as follows
	\[\dfrac{1}{(r+1)n +1} \binom{(r+1)n +1}{n} = \dfrac{1}{n} \binom{(r+1)n}{n-1} = \dfrac{1}{rn+1} \binom{(r+1)n }{n}. \] 
	Similar to the Catalan numbers we can give a nice combinatorial interpretation of these three formulas as follows:
\begin{theorem} \label{th11} \
\begin{enumerate}
\item The number of paths in $\mathcal{P}(n,r,1)$ that start with an up step with exactly $k$ up steps starting on or 
	below the $x$-axis for $k = 1, 2, \dots, r n + 1$ is given by $\tfrac{1}{rn+1} \binom{(r+1)n }{n}$.
 
\item The number of paths in $\mathcal{P}(n,r,1)$ that start with a down step with exactly $k$ down steps that start 
	on or below the $x$-axis for $k = 1, 2, \dots, n$ is given by $\tfrac{1}{n} \binom{(r+1)n}{n-1}$. 
 
\item The number of paths in  $\mathcal{P}(n,r,1)$ with exactly $k$ vertices on or below the $x$-axis 
	for $k = 1, 2, \dots, (r+1)n + 1$ is given by $\tfrac{1}{(r+1)n +1} \binom{(r+1)n +1}{n}$.
\end{enumerate}
\end{theorem}
 
\begin{proof}
	Here we can apply Theorem \ref{spver} again to prove the above statements. Let $p$ be any path in $\mathcal{P}(n,r,1)$. To prove (1) we take the initial vertices of the up steps of $p$ as our special vertices. Since there are $r n + 1$ up steps, $p$ has $r n+1 $ conjugates that start with an up step. By Theorem \ref{spver} there is exactly one conjugate of $p$ with exactly $k$ up steps starting on or below the $x$-axis and we know that the number of paths in $\mathcal{P}(n,r,1)$ that start with an up step is given by 	the binomial coefficient $\binom{(r+1)n}{n}$. Therefore the number of paths starting with an up step and 	having $k$ up steps on or below the $x$-axis is given by $\tfrac{1}{rn+1} \binom{(r+1)n }{n}$. The proofs of parts two and three follow similarly, taking the initial vertices of the down steps as special vertices and all vertices as special vertices, respectively.
\end{proof}

	It is noteworthy to mention here the following corollary which is the classical analogue of the generalized 	version of the Chung-Feller theorem.
\begin{cor}
	The number of paths in $\mathcal{P}(n,r,0)$, with exactly $k$ up steps below the $x$-axis is independent of 	$k$ for $k = 0, 1, 2, \dots, r n$ and is given by $\tfrac{1}{rn+1} \binom{(r+1)n }{n}$.
\end{cor}

\begin{proof}
	The technique used to prove the classical version applies here too. If we remove the first up step of the paths 	in Theorem \ref{th11}(1) and shift them down one level then the up steps on or below the $x$-axis becomes 	up steps below the $x$-axis and we have paths in $\mathcal{P}(n,r,0)$ that satisfy the corollary. 
\end{proof}
\clearpage
\section{The generating function approach}\label{sec:thirteenth}
We can also prove Theorem \ref{th11} using generating functions. In this section we'll give a sketch of the proof of Theorem \ref{th11}(1) and \ref{th11}(2) and omit the proof of \ref{th11}(3).

Let $t$ be the weight on the up steps. Then the generating function for paths in $\mathcal{P}(n,r,0,+)$ has the form
\begin{equation} \label{e6.4}
	C_g(y, t) = \sum_{n=0}^{\infty} C_n^r y^{n} t^{r n}.
\end{equation}
 So \[ C_g(y, t)=f(y t^r).\]	
Since we are interested in up steps that start on or below the $x$-axis we weight them by $s$ and the up steps that start above the $x$-axis will have weight $t$. We also distinguish between the down steps that start on or below the $x$-axis (weighted by $x$) and the down steps that start above the $x$-axis (weighted by $y$). 
It is not hard to show that any path in $\mathcal{P}(n,r,0)$ can be uniquely factored into three different types of primes by looking at each time they return to the $x$-axis, as shown in Figure \ref{rprimes}. 

Let us consider positive prime paths in $\mathcal{P}(n,r,0)$, i.e., paths that stay above the $x$-axis and do not return to the $x$-axis till the end.  If we remove the last down step from one of these prime paths we get a path in $\mathcal{P}(n,r,r,+)$ that stay strictly above the $x$-axis. Therefore the generating function of the positive prime paths in $\mathcal{P}(n,r,0)$ is 
\begin{equation} \label{eq6.5}
	F_{p} = y s C_g(y,t) (t C_g (y, t))^{r-1}.
\end{equation}
Similarly the generating function of the negative prime paths in $\mathcal{P}(n,r,0)$ is 
\begin{equation} \label{eq6.6}
	F_{n} = x (s C_g(x, s))^r.
\end{equation}
	We also have another type of prime path in $\mathcal{P}(n,r, 0)$ that contains a down step that 
	\begin{figure}[h!]
	\centering
		\includegraphics[width=0.90\textwidth]{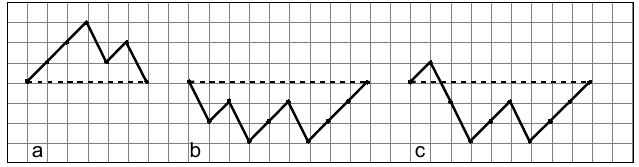}
	\caption{Primes in $\mathcal{P}(n,2,0)$. (a) A positive prime,  (b) a negative prime, and (c) a mixed prime.}
	\label{rprimes}
\end{figure}
cross the $x$-axis and is of the form $p_{r-i} D q_{i}$ for $i=1, \dots, r-1$  where $p_{r-i}$ is a strictly positive path in $\mathcal{P}(n,r, r-i,+)$ that starts on the $x$-axis and end at height $r-i$, $D$ is a down step and $q_{i}$ is a strictly negative path that start from height $-i$ and touches the $x$-axis at the end. These differ from the positive and negative primes because of the down step $D$ that crosses the $x$-axis. We call these mixed primes. The mixed primes have the generating function 
\begin{equation} \label{eq6.7}
	F_{m} = \sum_{i=1}^{r-1} s C_g(y,t) (t C_g (y, t))^{r-1-i} y (s C_g(x, s))^i.
\end{equation}
Note that we can combine the generating function of the positive primes and the mixed primes in to one formula as
\begin{equation} \label{eq6.71}
	F_{p+m} = \sum_{i=0}^{r-1} s C_g(y,t) (t C_g (y, t))^{r-1-i} y (s C_g(x, s))^i
\end{equation}
and the functional equation for $f(x s^r)$ is 
\[s f(x s^r)=s+ x s^{r+1} f^{r+1}(x s^r).\]
Which in terms of $C_g(x, s)$ is 
\[s C_g(x, s)= s + x (s C_g(x, s))^{r+1}.\]

	Using the generating function of the primes we express the generating function for the paths in $\mathcal{P}(n,r,0)$ (denoted by $G_{0} (x, s, y, t)$) as
\begin{equation} \label{eq6.8}
	G_{0} (x, s, y, t) = \frac{1}{1- F_{n} -F_{p+m}}.
\end{equation}
	Since any path in $\mathcal{P}(n,r,1)$ can be uniquely decomposed into two parts where the first part is a 
	path in $\mathcal{P}(n,r,0)$ and the last part is a path in $\mathcal{P}(n,r,1,+)$ using \eqref{eq6.8} we 
	can write the generating function for the paths in $\mathcal{P}(n,r,1)$ (denoted by $G_{1}(x, s, y, t)$) as 
\begin{equation} \label{eq6.9}
	G_{1} (x, s, y, t) = \frac{s C_g(y,t)}{1 - F_{n} -F_{p+m}}.
\end{equation}

	If we consider the paths in $\mathcal{P}(n,r,1)$ that start with a down step they must start with a negative 
	prime. Therefore the generating function is
\begin{equation} \label{eq7.1}
	G_{1D} (x, s, y, t) = F_{n} G_{1}(x, s, y, t) = \frac{s C_g(y,t) F_{n}}{1 - F_{n} -F_{p+m}}.
\end{equation} 

On the other hand the paths in $\mathcal{P}(n,r,1)$ that start with an up step must start with a positive prime or a mixed prime. Therefore the generating function is 
\begin{equation} \label{eq7}
	G_{1U} (x, s, y, t) = G_{1}(x, s,y, t) - F_{n} G_{1}(x, s,y, t) = \frac{(1-F_n ) s C_g(y,t)}{1- F_{n} -F_{p+m}}.
\end{equation}

To show that the coefficient of \eqref{eq7.1}, \eqref{eq7} are the generalized Catalan numbers we need the following identities which are easy to prove.

\begin{equation}  \label{eq7.2}	
\begin{aligned}
	\sum_{n=0}^{\infty} C_n^r \sum_{i=0}^{r n} t^i s^{r n-i} &=\dfrac{t C_g(1,t) - s C_g(1,s)}{ t - s}\\ 
	\sum_{n=0}^{\infty} C_{n+1}^r \sum_{i=0}^{ n} x^i y^{n-i} &= \dfrac{C_g(x, 1) - C_g(y, 1)}{ x -y}
\end{aligned}
\end{equation}

The following relations are equivalent to Theorem \ref{th11}(1) and \ref{th11}(2).
\begin{equation}
\begin{aligned}
	G_{1U}(1,s,1,t)  &= \sum_{n=0}^{\infty} C_n^r \sum_{i=0}^{r n} s^{i+1} t^{r n-i}\\  
	G_{1D}(x,1,y,1) &= \sum_{n=0}^{\infty} C_{n+1}^r \sum_{i=0}^{ n} x^{i+1} y^{n-i}.
\end{aligned}
\end{equation}
	Similar to the proof of Theorem \ref{e13} we can algebraically show the following results
\begin{equation}
\begin{aligned} 
	\frac{(1-F_n ) s C_g(1,t)}{1- F_{n} -F_{p+m}} &=  \dfrac{(1-(s C_g(1, s))^r ) s C_g(1,t)}{1 - (s C_g(1, s))^r- \sum_{i=0}^{r-1} s C_g(1,t) (t C_g (1, t))^{r-1-i} (s C_g(1, s))^i} \\
	&= s \dfrac{ t C_g(1,t^r) - s C_g(1,s^r) }{ t - s}  
\end{aligned}
\end{equation}
and 
\begin{equation}
\begin{aligned} 
	  \frac{ C_g(y,1) F_{n}}{1 - F_{n} -F_{p+m}} &=  \dfrac{ C_g(y,1) x ( C_g(x, 1))^r}{1-x ( C_g(x, 1)^r - \sum_{i=0}^{r-1}  C_g(y,1) ( C_g (y, 1))^{r-1-i} y ( C_g(x, 1))^i}\\
	  &= x \dfrac{C_g(x, 1) - C_g(y, 1)}{ x -y}.
\end{aligned}
\end{equation}

	Therefore the coefficient of $t^{i+1} s^{r n-i}$ in $G_{1U}$ is independent of 
	$i$ for $0 \le i \le r n$ and is $C_n^r = \tfrac{1}{rn+1} \binom{(r+1)n }{n}$ and the coefficient of $x^{j+1} y^{n-j}$ in $G_{1D}$  is independent of $j$ for $0 \le j \le n$ and is $C_n^r = \tfrac{1}{n} \binom{rn}{n-1}$ respectively.
	\clearpage
\section{Versions of generalized Catalan number formula 2}\label{sec:forteenth}
In this section we show another way to generalize the Catalan number formula. If we consider the up steps to have weight $1$ and the down steps to have weight $x$ then $x f^{r}(x)$ is the generating function of the positive primes in $\mathcal{P}(n,r, 0)$. It is interesting to see that the generating function $x f^{r}(x)$ can be expressed in a different way by rewriting \eqref{e6.1} as  
\begin{equation} \notag
	x f^{r}(x) = 1 - f^{-1}(x).
\end{equation}
	Then the Lagrange inversion formula gives
\begin{align} \notag
	 f^{-1}(x) &= \sum_{n=0}^{\infty} \dfrac{-1}{(r+1)n-1} \binom{(r+1)n-1}{n} x^{n} \\ \notag
 			&= 1- \sum_{n=1}^{\infty} \dfrac{1}{(r+1)n-1} \binom{(r+1)n-1}{n} x^{n}
\end{align}
	So we get 
    \begin{equation} \notag
    x f^{r}(x) = \sum_{n=1}^{\infty} \dfrac{1}{(r+1)n-1} \binom{(r+1)n-1}{n} x^{n}.
    \end{equation}
	Note that  for $r=1$ the coefficients are just the Catalan 
	numbers. But for $r>1$, these are not the same as the coefficients in \eqref{e6.3}.

	These numbers can also be written in three different forms as follows:
	\[\dfrac{1}{(r+1)n -1} \binom{(r+1)n -1}{n} = \dfrac{1}{n} \binom{(r+1)n - 2}{n-1} = \dfrac{1}{r n-1} \binom{(r+1)n - 2}{n}. \] 

	Given such a prime path if we remove the first step (an up step) and shift the path down one level then we 
	have a path that starts at the origin and ends at height $-1$, i.e., paths in $\mathcal{P}(n,r,-1)$. Each of 
	these paths has $r n -1$ up steps and $n$ down steps. So we can use the cycle method to get the following 
	Chung-Feller theorems for them.

\begin{theorem} \label{th12} \ 
\begin{enumerate}
 \item The number of paths in $\mathcal{P}(n,r,-1)$ that start with an up step with exactly $k$ up steps starting on 
 	or above the $x$-axis for $k = 1, 2, \dots, r n - 1$ is given by $\tfrac{1}{r n-1} \binom{(r+1)n-2}{n}$.
 
  \item The number of paths in $\mathcal{P}(n,r,-1)$ that start with a down step with exactly $k$ down steps that 
  	start on or above the $x$-axis for $k = 1, 2, \dots, n$ is given by $\tfrac{1}{n} \binom{(r+1)n-2}{n-1}$. 
 
 \item The number of paths in  $\mathcal{P}(n,r,-1)$ with exactly $k$ vertices on or above the $x$-axis for 
 	$k = 1, 2, \dots, (r+1)n - 1$ is given by $\tfrac{1}{(r+1)n -1} \binom{(r+1)n -1}{n} $.
\end{enumerate}
\end{theorem}
 
\begin{proof}
	If we reflect the paths in $\mathcal{P}(n,r,-1)$ about the $x$-axis we get paths starting at the origin and 
	ending at height $1$ with steps that go up by $r$ and down by $1$, i.e. paths with step set 
	$\{ (1, r), (1, -1) \}$. Let us denote the set of such paths by $\mathcal{P}^*(n,r,1)$. Now the above 
	statements can be restated as
\begin{enumerate} 
\item The number of paths in $\mathcal{P}^*(n,r,1)$ that start with a down step with exactly $k$ down steps that 
	start on or below the $x$-axis for $k = 1, 2, \dots, r n-1$ is given by  $\tfrac{1}{r n-1} \binom{(r+1)n-2}{n}$.
\item The number of paths in $\mathcal{P}^*(n,r,1)$ that start with an up step with exactly $k$ up steps starting 
	on or below the $x$-axis for $k = 1, 2, \dots, n$ is given by $\tfrac{1}{n} \binom{(r+1)n-2}{n-1}$. 
\item The number of paths in  $\mathcal{P}^*(n,r,1)$ with exactly $k$ vertices on or below the $x$-axis for 
	$k = 1, 2, \dots, (r+1)n - 1$ is given by $\tfrac{1}{(r+1)n -1} \binom{(r+1)n -1}{n} $.
\end{enumerate}
	So the proof also follows from Theorem \ref{spver} with the same reasoning as the proof of Theorem \ref{th11}.
\end{proof}

In the next two section we'll consider two types of generalization of the Narayana numbers. The classical Narayana numbers are represented as a product of two binomial coefficients. Considering paths in $\mathcal{P}(n,r,h)$ we can generalize them either as a product of two binomial coefficients or as a product of $r+1$ binomial coefficients. We'll consider the former one in the next setion and the later in section \ref{sec:sixteenth}. We'll give combinatorial interpretation of both generalizations. 
\clearpage
\section{Peaks and valleys}\label{sec:fifteenth}

It is natural to ask how many paths are there in $\mathcal{P}(n,r,h)$ with a given number of peaks and valleys.
To do that first we look at non-negative paths in $\mathcal{P}(n,r,0)$, i.e. paths that stay weakly above the $x$-axis. Let $F$ be the generating function for the non-negative paths in $\mathcal{P}(n,r,0)$ with weight $x$ on the down steps and weight $t$ on the peaks defined by 
\begin{equation} \notag
F = \sum_n x^{n} t^{\pk} 
\end{equation}
where $\pk$ stands for number of peaks. If $P$ is the generating function for the primes then we can write $F$ in terms of $P$ as 
\begin{equation} \label{primes}
F=\dfrac{1}{1-P}.
\end{equation}  
Now the total number of peaks in such a path is the sum of the peaks of its primes. Let $p$ be such a prime path. So we decompose $p$ in the following way: Since $p$ does not return to the $x$-axis till the end we look at the last time it leaves height one, height two, and so on. The first part of the path consists of an up step followed by a positive path, the second part of the path consists of an up step followed by another positive path, \dots, and the last part consists of a path that starts with an up step and ends at height $r-1$ followed  by a down step. So the prime $p$ can be factored as \[p = U p_1 U p_2 \cdots U p_r D\] where $p_i$ is a positive path in $\mathcal{P}(n_i,r,0)$ for some $n_i$. The number of peaks of $p$ is the sum of the number of peaks of the $p_i$'s plus one more if $p_r$ is empty.
\begin{figure}[h]
	\centering
		\includegraphics[width=0.80 \textwidth] {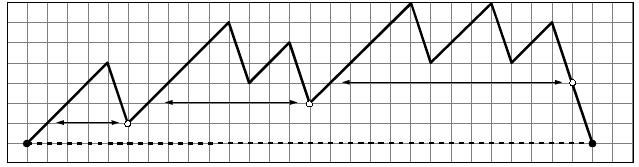}
	\caption{A prime path for $r=3$}
	\label{fig:psteppath}
\end{figure}
So we can write the generating function of the primes as
\begin{equation} \notag
P=x F^{r-1} (F-1+t)
\end{equation}
which together with \eqref{primes} gives the functional equation
\begin{equation} \notag
F = 1 + x F^r(F-1+t). 
\end{equation}
Setting $F=1+G$ and replacing $P$ by $\tfrac{G}{1+G}$ we get 
\begin{equation} \notag
G=x(1+G)^{r}(t+G).
\end{equation}
Applying the second form of the Lagrange inversion given at the beginning of this chapter and taking 
$\phi(G)=1+G$ and $g(u)=(1+u)^{r}(t+u)$ we get 
\begin{equation} \notag
\begin{aligned}
F &=1+G = \sum_n \left[u^n \right] \left(1- \dfrac{u((1+u)^{r}(t+u))'}{(1+u)^{r}(t+u)}\right) (1+u)((1+u)^{r}(t+u))^n x^n\\
   &= \sum_{n,k} \dfrac{1}{r n -k+1}\binom{r n}{k} \binom{n-1}{k-1} t^k x^n.
\end{aligned}
\end{equation}

We denote these coefficients by $N_r(n,k)= \tfrac{1}{r n -k+1}\binom{r n}{k} \binom{n-1}{k-1}$. Note that $N_1(n,k)$ gives us our familiar Narayana numbers. We also find that these coefficients can be written in five different forms as 
\begin{align*}
N_r(n,k)&= \dfrac{1}{n}\binom{r n}{k-1} \binom{n}{k} = \dfrac{1}{r n-k+1}\binom{r n}{k} \binom{n-1}{k-1}\\
	&=\dfrac{1}{k}\binom{r n}{k-1} \binom{n-1}{k-1} = \dfrac{1}{r n+1}\binom{r n+1}{k} \binom{n-1}{k-1}\\
	&=\dfrac{1}{n-k}\binom{r n}{k-1} \binom{n-1}{k}.
\end{align*}
Note that just like the relation \eqref{nara} $N_r(n,k)$ and $C_n^r$ are also related by the equation
\begin{equation} \notag
\sum_{k=1}^{n} \dfrac{1}{n}\binom{r n}{k-1} \binom{n}{k} = \dfrac{1}{r n+1}\binom{(r+1) n}{n}. 
\end{equation}
As a generalization of Theorem \ref{t6} we can use the cycle method to give a combinatorial interpretation of each of these forms as well.
\begin{theorem} \ 
\begin{enumerate}
\item  The number of paths in  $\mathcal{P}(n, r, 1)$ with $k-1$ peaks that start with a down step and end with an up step with exactly $j$ peaks on or below the $x$-axis for $j = 0, 1, 2, \dots, k-1$ is given by $\tfrac{1}{k}\binom{r n}{k-1} \binom{n-1}{k-1} $.

\item The number of paths in  $\mathcal{P}(n, r, 1)$ with $k-1$ valleys
	that start with an up step and end with a down step with exactly 
	$j$ valleys on or below the $x$-axis for $j=0, 1, 2, \dots, k-1$  is given by 
    	$\tfrac{1}{k} \binom{r n}{k-1} \binom{n-1}{k-1} $.

\item The number of paths in  $\mathcal{P}(n, r, 1)$ with 
	$n-k$ double rises that start with an up step and end with an up step
with exactly $j$ double 
	rises on or below the $x$-axis for $j=0, 1, 2, \dots, r n-k$ is given by 
	$\tfrac{1}{r n-k+1} \binom{r n}{k}  \binom{n-1}{k-1}$.

\item The number of paths in  $\mathcal{P}(n, r, 1)$ with 
 	$n-k-1$ double falls that start with a down step and end with a down
step with exactly 
 	$j$ double falls on or below the $x$-axis for $j=0, 1, 2, \dots, n-k-1$ is given by 
 	$\tfrac{1}{n-k} \binom{r n}{k-1}  \binom{n-1}{k}$.

\item The number of paths in  $\mathcal{P}(n, r, 1)$ with $k$ peaks that start with 
	an up step with exactly $j$ up steps on or below the $x$-axis for $j= 1, 2, \dots, r n+1$ 
	is given by $\tfrac{1}{r n+1}\binom{r n+1}{k} \binom{n-1}{k-1}$.

\item The number of paths in  $\mathcal{P}(n, r, 1)$ with $k$ valleys that start 
	with a down step with exactly $j$ down steps on or below the $x$-axis for 
	$j= 1, 2, \dots, n$ is given by $\tfrac{1}{n}\binom{r n}{k-1} \binom{n}{k}$.
\end{enumerate}
\end{theorem}

The proof is essentially same as the proof of Theorem \ref{t6} and uses similar arguments. Therefore we leave it to the reader. 
\clearpage
\section{A generalized Narayana number formula}\label{sec:sixteenth}

	Next we'll look at generalizations of the Narayana numbers for paths in  $\mathcal{P}(n-1,r, r+1)$. The interpretation we have found so far of the Narayana numbers in terms of Dyck paths is keeping track of up steps in even positions. In this section we are going to generalize this interpretation for paths in $\mathcal{P}(n-1,r, r+1)$. 

	The simplest way to approach this is to weight an up step in a position congruent to $i$ modulo $r+1$
	by $\alpha_{i}$, $i \ge 0$. To do this we take the steps $r+1$ at a time, i.e., we replace the original set of 
	two steps (up by $1$ and down by $r$) with a new set of $2^{r+1}$ steps that correspond to all possible paths 
	made of $r+1$ of the original steps. Note that each new step goes up (or down) by a multiple of $r+1$. 

	Now let us define a generating function for these new steps to be 
	\[ \mathcal{F} = \sum_{\sigma \in S} w(\sigma) t^{d(\sigma)},\] 
	where $S$ is the set of $2^{r+1}$ new steps, $w(\sigma)$ is the weight of $\sigma$ (defined in terms of the 
	original up steps comprising $\sigma$, where an up step in a position congruent to $i$ modulo $r+1$ has 
	weight $\alpha_{i}$) and $d(\sigma)$ is $\tfrac{1}{r+1}$ times the distance down that the new step goes. A new step that goes up by $r+1$ is interpreted as going down by $-(r+1)$.

	It is not hard to see that this generating function can be written as
\begin{equation}
	\mathcal{F} = t^{-1}(\alpha_{0} + t)(\alpha_{1} + t) \cdots (\alpha_{r} + t).
\end{equation}
	For example, when $r=1$ the $2^{2} = 4$ new steps are shown in Figure \ref{fig:stepset}.	
\begin{figure}[ht]
	\centering
		\includegraphics[width=0.80 \textwidth] {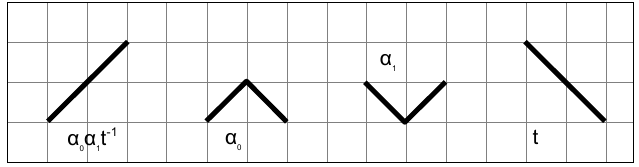}
	\caption{Step set for $r=1$}
	\label{fig:stepset}
\end{figure}

	To find a generating function for these paths we consider a more general situation. Suppose we want to count paths with steps that go up by $1$ or down by any nonnegative integer. We weight a step that goes down by $i$ with the weight $w_{i+1}$, where we think of an up step as a step that goes down by $-1$. Let $W(t) = \sum_{i=0} w_i t^i$. Let $F$ be the generating function for paths that stay strictly above the $x$-axis 
	and end at height $r+1$. Then the generating function for strictly positive paths that end at height $h(r+1)$ is $F^h$ 
	(including $h=0$) and by removing the last step of a path counted by $F$, we see that $F$ satisfies 
	$F=W(F)$.

	Returning to our original problem, we see that the generating function $F$ for paths weighted according to 
	the positions modulo $r+1$ of the up steps satisfies 
\begin{equation} \label{pmodr}
	F = (\alpha_0 + F)(\alpha_1 + F) \cdots (\alpha_r + F),
\end{equation} 
Applying Lagrange inversion to \eqref{pmodr} we get 
\begin{equation} \notag
	F^h = \sum_n \sum_{n_0 + \cdots + n_r = (n-1)r + h(r+1)} \frac{h}{n} \binom{n}{n_0} \cdots \binom{n}{n_r} \alpha_0^{n_0} \alpha_1^{n_1} \dots \alpha_r^{n_r}.
\end{equation}
So for $h=1$ we get 
 \begin{equation} \notag
F 	=  \sum_n \sum_{n_0 + \cdots + n_r = n r + 1} \frac{1}{n} \binom{n}{n_0} \cdots \binom{n}{n_r} \alpha_0^{n_0} \alpha_1^{n_1} \dots \alpha_r^{n_r}.
\end{equation}

	We denote these coefficient by $N^r(n, n_0, n_1, \dots, n_r) = \frac{1}{n} \binom{n}{n_0} \cdots \binom{n}{n_r} $, where 
	$n_0 + \cdots + n_r = n r + 1$. Consider the case $r=1$. Then $N^1(n, n_0, n_1)= \frac{1}{n} \binom{n}{n_0} 	\binom{n}{n_1} = \frac{1}{n} \binom{n}{n_0} \binom{n}{n_0-1}$ is our well known Narayana number. We 
	have given a combinatorial interpretation for the number $N^1(n,n_0, n_1)$ in Theorem \ref{th91}. Here we state a generalization of Theorem \ref{th91} using congruence. 
\begin{theorem} \label{nrnara} \ 
\begin{enumerate}
\item The number of paths in $\mathcal{P}(n-1,r,r+1)$ that start with a down step having exactly $n_i-1$ 
	down steps starting at positions congruent to $i \pmod{r+1}$ for each $i=0, \dots, r$, with $j$ down steps starting at positions congruent to $0 \pmod{r+1}$ on or below the $x$-axis is independent of $j$, $j=1, \dots, n_0-1$ and is given by the numbers $N^r(n, n-n_0+1,  \dots, n-n_r+1) = \tfrac{1}{n_0-1} \binom{n-1}{n_0-2}\binom{n}{n_1-1}  \cdots \binom{n}{n_r-1}$.
	
\item The number of paths in $\mathcal{P}(n-1,r,r+1)$ that start with an up step having exactly $n_i$ 
	up steps starting at positions congruent to $i \pmod{r+1}$ for each $i=0, \dots, r$, with $j$ up steps starting at positions congruent to $0 \pmod{r+1}$ on or below the $x$-axis is independent of $j$, $j=1, \dots, n_0$ and is given by the numbers $N^r(n,n_0,  \dots, n_r) = \tfrac{1}{n_0} \binom{n-1}{n_0-1}\binom{n}{n_1}  \cdots \binom{n}{n_r}$.
\end{enumerate}
\end{theorem}

\begin{proof}
We'll prove Theorem \ref{nrnara}(2) using the cycle method. Consider paths in $\mathcal{P}(n-1,r,r+1)$. They have $nr+1$ up steps and $n-1$ down steps, so a total of $n(r+1)$ steps. There are $n$ positions congruent to $i$ modulo $r+1$ for each $i=0,\dots,r$. So there are $\binom{n-1}{n_0-1}\binom{n}{n_1}  \cdots \binom{n}{n_r}$ paths in $\mathcal{P}(n-1,r,r+1)$ having $n_i$ up steps at positions congruent to $i$ modulo $r+1$ that start with an up step. Each such path has $n_0$ conjugates that start with an up step at positions that are a multiple of $r+1$. To apply our cycle method we convert these paths to $r+1$ colored free Motzkin paths by taking $r+1$ steps at a time, where an $r+1$ colored free Motzkin path is a path that lie in the half plane $\mathbb{Z}_{\ge 0}\times \mathbb{Z}$ having unit up steps $(1,1)$, unit down steps $(1,-1)$, and $r+1$ different colored unit flat steps. The resulting $r+1$ colored free Motzkin paths will end at height $1$.

Therefore as in the proof of Theorem \ref{th91}, we can deduce that the total number of paths in $\mathcal{P}(n-1,r,r+1)$ that start with an up step having exactly $n_i$ up steps starting at positions congruent to $i \pmod{r+1}$
with $j$ up steps on or below the $x$-axis is $\tfrac{1}{n_0} \binom{n-1}{n_0-1}\binom{n}{n_1}  \cdots \binom{n}{n_r}$. 

Similarly we get the result of Theorem \ref{nrnara}(1) by considering paths in $\mathcal{P}(n-1,r,r+1)$ that start with a down step having $n_i-1$ down steps at positions congruent to $i$ modulo $r+1$, where $\sum_{i=0}^{r} (n_i-1) = n-1$. 
\end{proof}

Note that when $r=1$ in Theorem \ref{nrnara}(1), the number of paths in $\mathcal{P}(n-1,1,2)$ that start with a down step having exactly $n_i-1$ down steps starting at positions congruent to $i \pmod{r+1}$ for $i=0,1$, with $j$ down steps congruent to $0 \pmod{r+1}$ on or below the $x$-axis is  $\tfrac{1}{n_0-1} \binom{n-1}{n_0-2}\binom{n}{n_1-1}$. Since the total number of down step is $n_0-1 + n_1-1 = n-1$, we have $n_1= n-n_0+1$, So the total number of paths is $N^1(n, n-n_0+1, n-n_1+1) = \tfrac{1}{n_0-1} \binom{n-1}{n_0-2}\binom{n}{n_1-1} = \tfrac{1}{n_0-1} \binom{n-1}{n_0-2}\binom{n}{n_0}$, which is exactly what we had in Theorem \ref{th91}(1).

\nocite{*}
\backmatter
\singlespacing
\bibliographystyle{acm}	
\bibliography{Diss_bib}		

\begin{thebibliography}{10}

\bibitem{ma}
{\sc Aguiar, M., and Moreira, W.}
\newblock Combinatorics of the free {B}axter algebra.
\newblock {\em Electron. J. Combin. 13\/} (2006), R17.

\bibitem{sparre}
{\sc Andersen, E.~S.}
\newblock The equivalence principle in the theory of fluctuations of sums of
  random variables.
\newblock {\em Colloquium on Combinatorial Methods in Probability Theory\/}
  (1962), 13--16.

\bibitem{fuss}
{\sc Aval, J.-C.}
\newblock Multivariate {F}uss-{C}atalan numbers.
\newblock {\em Disc. Math. 308\/} (2008), 4660--4669.

\bibitem{baxter}
{\sc Baxter, G.}
\newblock Combinatorial method in fluctuation theory.
\newblock {\em Z. Wahrscheinlichkeitstheorie 1\/} (1963), 263--270.

\bibitem{motz}
{\sc Bernhart, F.~R.}
\newblock Catalan, {M}otzkin and {R}iordan numbers.
\newblock {\em Disc. Math. 204\/} (1999), 73--112.

\bibitem{tabc}
{\sc Buontempo, J., and Hopkins, B.}
\newblock Tableau cycling and {C}atalan numbers.
\newblock {\em Electron. J. Combin. 7\/} (2007), A45.

\bibitem{dc}
{\sc Callan, D.}
\newblock Pair them up! {A} visual approach to the {C}hung-{F}eller theorem.
\newblock {\em College Mathematics Journal 26\/} (1995).

\bibitem{dcal}
{\sc Callan, D.}
\newblock Why is $ \tfrac{1}{n+1} \binom{2n}{n} = \tfrac{1}{2n+1}
  \binom{2n+1}{n} = \tfrac{1}{n} \binom{2n}{n-1}$.
\newblock \url{http://www.stat.wisc.edu/~callan/notes}, 2004.

\bibitem{bb}
{\sc Callan, D.}
\newblock A combinatorial interpretation of $\tfrac{j}{n} \binom{kn}{n+j}$.
\newblock \url{http://www.citebase.org/abstract?id=oai:arXiv.org:math/0604471},
  2006.

\bibitem{gg}
{\sc Chen, Y.-M.}
\newblock The {C}hung-{F}eller theorem revisited.
\newblock {\em Disc. Math. 308\/} (2008), 1328--1329.

\bibitem{hh}
{\sc Chung, K.~L., and Feller, W.}
\newblock On fluctuations in-coin tossing.
\newblock {\em Proc. Natl. Acad. Sci. 35\/} (1949), 605--608.

\bibitem{linv}
{\sc Comtet, L.}
\newblock {\em Advanced Combinatorics}.
\newblock Reidel, Boston, MA, 1974.

\bibitem{cori}
{\sc Cori, R.}
\newblock Words and trees.
\newblock In {\em Combinatorics on Words}, M.~Lothaire, Ed. Cambridge
  University Pres, New York, NY, 1983, pp.~215--229.

\bibitem{zak}
{\sc Dershowitz, N., and Zaks, S.}
\newblock The cycle lemma and some applications.
\newblock {\em Europ. J. Combinatorics 11\/} (1990), 35--40.

\bibitem{aa}
{\sc Deutsch, E.}
\newblock Dyck path enumeration.
\newblock {\em Disc. Math. 204\/} (1999), 167--202.

\bibitem{sch3}
{\sc Deutsch, E.}
\newblock A bijective proof of the equation linking the {S}chr\"oder numbers,
  large and small.
\newblock {\em Disc. Math. 241\/} (2001), 235--240.

\bibitem{tcolor}
{\sc Deutsch, E., and Shapiro, L.~W.}
\newblock A bijection between ordered trees and 2-{M}otzkin paths, and its many
  consequences.
\newblock {\em Disc. Math. 256\/} (2002), 655--670.

\bibitem{ll}
{\sc Dvoretzky, A., and Motzkin, T.}
\newblock A problem of arrangements.
\newblock {\em Duke Math. J. 14\/} (1947), 305--313.

\bibitem{ff}
{\sc Eu, S.-P., Fu, T.-S., and Yeh, Y.-N.}
\newblock Refined {C}hung-{F}eller theorems for lattice paths.
\newblock {\em J. Combin. Theory Ser. A 112\/} (2005), 143--162.

\bibitem{ii}
{\sc Eu, S.-P., Liu, S.-C., and Yeh, Y.-N.}
\newblock Taylor expansions for {C}atalan and {M}otzkin numbers.
\newblock {\em Adv. Appl. Math. 29\/} (2002), 345--357.

\bibitem{foata}
{\sc Foata, D., and Sch\"utzenberger, M.~P.}
\newblock On the principle of equivalence of {S}parre {A}nderson.
\newblock {\em Math. Scand. 28\/} (1971), 308--316.

\bibitem{qlag}
{\sc Gessel, I.}
\newblock A noncommutative generalization and $q$-analog of the {L}agrange
  inversion formula.
\newblock {\em Trans. Amer. Math. Soc. 257\/} (1980), 455--482.

\bibitem{gessel}
{\sc Gessel, I.~M.}
\newblock A combinatorial proof of the multivariable {L}agrange inversion
  formula.
\newblock {\em J. Combin. Theory Ser. A 45\/} (1987), 178--195.

\bibitem{gould}
{\sc Gould, H.~W.}
\newblock {\em Catalan and Bell Numbers: Research bibliography of two special
  number sequences, Mathematica Monongaliae}, vol.~12.
\newblock Combinatorial Research Institute, Morgantown, WV, 1977.

\bibitem{oo}
{\sc Hilton, P., and Pedersen, J.}
\newblock Catalan numbers, their generalization, and their uses.
\newblock {\em Math. Intelligencer 13\/} (1991), 64--75.

\bibitem{jr}
{\sc Jewett, R.~I., and Ross, K.~A.}
\newblock Random walks on $\mathbb{Z}$.
\newblock {\em College Mathematics Journal 19\/} (1988).

\bibitem{mm}
{\sc Kreweras, G.}
\newblock Joint distributions of three descriptive parameters of bridges.
\newblock In {\em Combinatoire \'enum\'erative}, A.~Dold and B.~Eckmann, Eds.,
  vol.~1234 of {\em Lecture Notes in Mathematics}. Springer, Berlin/Heidelberg,
  1986, pp.~177--191.

\bibitem{jj}
{\sc Labelle, J., and Yeh, Y.-N.}
\newblock Generalized {D}yck paths.
\newblock {\em Disc. Math. 82\/} (1990), 1--6.

\bibitem{ma-2009}
{\sc Ma, J., and Yeh, Y.-N.}
\newblock Generalizations of {C}hung-{F}eller theorem ii.
\newblock \url{http://www.citebase.org/abstract?id=oai:arXiv.org:0903.0705},
  2009.

\bibitem{mc}
{\sc MacMahon, P.~A.}
\newblock Memoir on the theory of the partitions of numbers, {P}art {IV}.
\newblock {\em Philos. Trans. Roy. Soc. London Ser. A 209\/} (1909).

\bibitem{moh}
{\sc Mohanty, S.~G.}
\newblock {\em Lattice Path Counting and Applications}.
\newblock Academic Press, New York, NY, 1979.

\bibitem{narayana}
{\sc Narayana, T.~V.}
\newblock A partial order and its applications to probability.
\newblock {\em Sankhya 21\/} (1959), 91--98.

\bibitem{kk}
{\sc Narayana, T.~V.}
\newblock Cyclic permutation of lattice paths and the {C}hung-{F}eller theorem.
\newblock {\em Skand. Aktuarietidskr. 50\/} (1967), 23--30.

\bibitem{raney}
{\sc Raney, G.~M.}
\newblock Functional composition patterns and power series reversion.
\newblock {\em Trans. Am. Math. Soc. 94\/} (1960), 441--451.

\bibitem{sch1}
{\sc Shapiro, L.~W., and Sulanke, R.~A.}
\newblock Bijections for the {S}chr\"oder numbers.
\newblock {\em Mathematics Magazine 73\/} (2000), 369--376.

\bibitem{nn}
{\sc Snevily, H.~S., and West, D.~B.}
\newblock The bricklayer problem and the strong cycle lemma.
\newblock {\em Amer. Math. Monthly 105\/} (1998), 131--143.

\bibitem{dd}
{\sc Stanley, R.~P.}
\newblock {\em Enumerative {C}ombinatorics}, vol.~2.
\newblock Cambridge University Press, New York, NY, 1999.

\bibitem{cc}
{\sc Sulanke, R.~A.}
\newblock Counting lattice paths by {N}arayana polynomials.
\newblock {\em Electron. J. Combin. 7\/} (2000), R40.

\bibitem{sch2}
{\sc Sulanke, R.~A.}
\newblock The {N}arayana distribution.
\newblock {\em J. Statist. Plann. Inference 101\/} (2002), 311--326.

\bibitem{higherdim}
{\sc Sulanke, R.~A.}
\newblock Generalizing {N}arayana and {S}chr\"oder numbers to higher
  dimensions.
\newblock {\em Electron. J. Combin. 11\/} (2003), R54.

\bibitem{ee}
{\sc Woan, W.-J.}
\newblock Uniform partition of lattice paths and {C}hung-{F}eller
  generalizations.
\newblock {\em Amer. Math. Monthly 108\/} (2001), 556--559.

\end{thebibliography}

\end{document}